\documentclass[10pt,a4paper,reqno]{amsart}
\usepackage{epsfig}
\usepackage{color}
\usepackage{amssymb,amsmath,amsthm,amstext,amsfonts}
\usepackage{amssymb,amsmath,amsthm,amstext,amsfonts}
\usepackage{amsmath,amstext,amsthm,amsfonts}
\usepackage{color}
\usepackage[mathscr]{eucal}
\usepackage{dsfont}

\usepackage{hyperref}
\RequirePackage[dvipsnames]{xcolor} 
\definecolor{halfgray}{gray}{0.55} 
\definecolor{webgreen}{rgb}{0,0.5,0}
\definecolor{webbrown}{rgb}{.6,0,0} \hypersetup{%
	colorlinks=true, linktocpage=true, pdfstartpage=3,
	pdfstartview=FitV,%
	breaklinks=true, pdfpagemode=UseNone, pageanchor=true,
	pdfpagemode=UseOutlines,%
	plainpages=false, bookmarksnumbered, bookmarksopen=true,
	bookmarksopenlevel=1,%
	hypertexnames=true,
	pdfhighlight=/O,
	urlcolor=webbrown, linkcolor=RoyalBlue,
	citecolor=webgreen, 
	pdftitle={},%
	pdfauthor={},%
	pdfsubject={2000 MAthematical Subject Classification: Primary:},%
	pdfkeywords={},%
	pdfcreator={pdfLaTeX},%
	pdfproducer={LaTeX with hyperref}%
}

\usepackage{psfrag}
\usepackage{url}
\usepackage{epstopdf}
\usepackage{epic,eepic}
\usepackage{upgreek}
\usepackage{amssymb}
\usepackage{mathrsfs}
\usepackage{verbatim}
\usepackage{mathrsfs}
\usepackage{amstext}
\usepackage{amsthm}
\usepackage{amssymb}
\usepackage{graphicx}

\theoremstyle{plain}
\newtheorem{theorem}{Theorem}[section]

\newtheorem{lemma}[theorem]{Lemma}

\theoremstyle{definition}
\newtheorem{remark}[theorem]{Remark}

\newcommand{\N}{\mathbb{N}}

\newcommand{\R}{\mathbb{R}}

\begin{document}
	
	\title[Metric mean dimension with potential of level sets]{Variational principles for metric mean dimension with potential of level sets}

	\author{Lucas Backes}
	\author{Chunlin Liu\textsuperscript{*}}
	\author{Fagner B. Rodrigues}
	\thanks{\textsuperscript{*}Corresponding author. Email: chunlinliu@mail.ustc.edu.cn}
	\address[Lucas Backes]{\noindent Departamento de Matem\'atica, Universidade Federal do Rio Grande do Sul, Av. Bento Gon\c{c}alves 9500, CEP 91509-900, Porto Alegre, RS, Brazil.}
	\email{lucas.backes@ufrgs.br}
	
	\address[Chunlin Liu]{\noindent School of Mathematical Sciences, Dalian University of Technology, Dalian, 116024, P.R. China}
	\email{chunlinliu@mail.ustc.edu.cn}
	
	\address[Fagner B. Rodrigues]{\noindent Departamento de Matem\'atica, Universidade Federal do Rio Grande do Sul, Av. Bento Gon\c{c}alves 9500, CEP 91509-900, Porto Alegre, RS, Brazil.}
	\email{fagnerbernardini@gmail.com}
	
	\date{\today}

	
	

	\begin{abstract}
		We establish three variational principles for the upper metric mean dimension with potential of level sets of continuous maps in terms of the entropy of partitions and Katok's entropy for dynamical systems exhibiting the specification property. 
		Moreover, we apply our results to investigate the metric mean dimension of suspension flows. As a byproduct, we establish certain properties of suspension flows and prove a measure-theoretic metric mean dimension version of Abramov's formula.
	\end{abstract}
	
	\maketitle

	\section{Introduction}
	Given a continuous map $f\colon X\to X$ defined on a compact metric space $(X,d)$, a continuous map $\varphi \colon X\to \mathbb R$ and $\alpha\in \mathbb R$, in this paper, we are interested in studying the size of the \emph{level sets}
	\begin{align*}
		K_\alpha=\left\{x\in X: \lim_{n\to\infty}\frac{1}{n}\sum_{j=0}^{n-1}\varphi(f^{j}(x))=\alpha\right\}
	\end{align*}
	by means of the \emph{upper metric mean dimension with potential}. This type of result is within the scope of \emph{multifractal analysis}.
	In very general terms, the idea of multifractal analysis consists in decomposing the phase space as the union of subsets of points with similar dynamical behavior, for instance, as the union of sets of points with the same Birkhoff average
	\[X=\bigcup_{\alpha\in \mathbb R}K_\alpha \cup I(\varphi)\]
	where $I(\varphi)$ is the set of points for which the Birkhoff average does not exist, and to describe the size of each such subset from a geometrical or topological viewpoint. The information (collection of numbers) obtained via this procedure for one such decomposition of the phase space is
	called a \emph{multifractal spectrum}. Then, the general idea is that if we know some properties of these spectra, we can recover useful information about the dynamics (see \cite{Barreira, BPS,BPS2, Cli, Olsen,Pesin, STVW,TV2}).
	
	In the works of Takens and Verbitskiy \cite{TV} and Thompson \cite{Tho,MR2765447}, the authors explored the spectrum derived from measuring the size of level sets of Birkhoff averages using the topological entropy and topological pressure. However,  there are plenty of systems with infinite topological entropy. For instance, they form a $C^0$-generic set in the space of homeomorphisms of a compact manifold with dimension greater than one \cite{Yano}. In particular, the results of \cite{TV, Tho,MR2765447} may not give useful information for systems in this big set. So, the main objective of the present note was to develop results in the spirit of these works which are better suited for this type systems. For this purpose, we use a generalization of the \emph{metric mean dimension} to study the size of $K_\alpha$.
	
	The metric mean dimension was introduced by Lindenstrauss and Weiss in \cite{LW} as metric-dependent analog of the \emph{mean dimension}, a topological invariant associated to a dynamical system introduced by Gromov \cite{Gromov}.  The mean dimension has several applications, in particular, it plays an important role in the study of embedding problems \cite{GuLinTsu2016,GuQiaoTsu2019,GuTsu2020}.
	Meanwhile, the metric mean dimension presents an upper bound to it, and if one furthers assume that the dynamics exhibits the marker property, then there exists a metric compatible with the topology such that the metric mean dimension with respect to this metric is equal to it \cite{GuLinTsu2016,Lin2000}. Furthermore, the metric mean dimension can be seen as a reparameterized version of the topological entropy. As a consequence, it follows readily from the definition that if the topological entropy of a system is finite then its metric mean dimension is zero while if the metric mean dimension is positive then the topological entropy is infinite. In particular, this is a well suited quantity to study systems with infinite topological entropy.
	
	In \cite{TSU2020}, Tsukamoto introduced the notion of  \emph{metric mean dimension with potential}, which is a generalization of the  metric mean dimension in the same way the topological pressure is a generalization of the topological entropy. 
	In this work we use precisely this notion to measure the size of the level sets $K_\alpha$ and relate them with ergodic aspects of the system $f$. More precisely, we establish three variational principles for this, articulated through the metric entropy of partitions and Katok's entropy, respectively (see Theorem \ref{thm:main1}).  Moreover, we apply these results to study the metric mean dimension  of suspension flows.  We recall that previous connections between the metric mean dimension and ergodic theory were given, for instance, by Lindenstrauss and Tsukamoto \cite{LT}, Velozo and Velozo \cite{VV}, Tsukamoto \cite{TSU2020},  Shi \cite{Shi}, Gutman and \'Spiewak \cite{GS}, Yang, Chen and Zhou \cite{YCZ} and by the first and third authors of the present paper \cite{BR}.

	After this work was completed, we learned that a result similar to our Theorem \ref{thm:main1}, was obtained simultaneously and independently by Zhang, Chen and Zhou \cite{ZCZ}. They established a variational principle for the level sets regarding Bowen/packing metric mean dimension and the metric entropy of partitions. However, we consider different definitions of measure-theoretical metric mean dimension. Additionally, we apply our results to suspension flows, which have attracted significant attention  (see, for example, \cite{GuJin2019,Gu,GuShi2024}). To enable the application of our results to suspension flows, we also derived some properties of suspension flows, such as Theorem \ref{thm:byproduct}. As a byproduct, we obtain a measure-theoretical metric mean dimension version of Abramov's formula (see Remark \ref{rem:Aformula}).

	\section{Preliminaries}
	Let $(X,d)$ be a compact metric space and $f \colon X \to X$ be a continuous map. Given $n\in \mathbb{N}$, we define the \emph{dynamical metric} or the 
	\emph{Bowen metric} $d_n \colon X \times X \, \to \,[0,\infty)$ by
	\[
	d_n(x,z)=\max\,\Big\{d(f^i(x),f^i(z)):0\leq i\leq n-1\Big\}.
	\]
	It is easy to see that $d_n$ is indeed a metric and, moreover, generates the same topology as $d$. Furthermore, given $\varepsilon > 0$, $n \in \mathbb{N}$ and a point $x \in X$, we define the open $(n, \varepsilon)$-ball around $x$ by
	\begin{displaymath}
		\mathit{B}_{n}(x, \varepsilon) = \{ y \in X ; d_n( x, y) < \varepsilon \}.
	\end{displaymath}
	We sometimes call these \emph{$(n, \varepsilon)$-dynamical balls} of radius $\varepsilon$ and length $n$. Given $K\subset X$, we say that $E \subset K$ is an \emph{$(n,\varepsilon)$--separated} subset of $K$ if $d_n(x,z) > \varepsilon$ for every $x,z \in E$ with $x\neq y$. Moreover, we say that a set $F \subset K$ is an \emph{$(n, \varepsilon)$-spanning} set of $K$ if $K \subset \bigcup_{x \in F} B_n(x, \varepsilon)$.

	\subsection{Metric mean dimension with potential}\label{sec: mmdim def}  We now present the notion of Bowen metric mean dimension with potential on non-compact sets introduced in \cite{CHeng}.
	Let $C(X,\mathbb{R})$ denote the set of all continuous maps from $X$ to $\mathbb{R}$. For any potential $\psi \in C(X, \mathbb{R})$, $Z\subset X$, $n\in \N$ and $\varepsilon>0$, define 
	\[
	m(Z,\psi,s, N,\varepsilon)=\inf_{\Gamma}\left\{\sum_{i\in I}\exp{\left(-sn_i+|\log \varepsilon|\sup_{x\in B_{n_i}(x_i,\varepsilon)}S_n \psi (x)\right)}\right\},
	\]
	where the infimum is taken over all covers $\Gamma=\{B_{n_i}(x_i,\varepsilon)\}_{i\in I}$ of $Z$ with $n_i\geq N$ and $S_n \psi (x)=\sum_{j=0}^{n-1}\psi (f^j(x))$.  We also consider
	\[
	m(Z,\psi,s,\varepsilon)=\lim_{N\to\infty}m(Z,\psi,s, N,\varepsilon).
	\]
	One can show (see for instance \cite{Pesin}) that there exists a certain number $s_0\in [0,+\infty)$ such that $m(Z,\psi, s,\varepsilon)=0$ for every $s>s_0$ and $m(Z,\psi, s,\varepsilon)=+\infty$ for every $s<s_0$. In particular, we may consider
	\[
	m\,\Big(Z,f,\psi,\varepsilon\Big)=\inf\{s:m(Z,\psi,s,\varepsilon)=0\}=\sup\{s:m(Z,\psi, s,\varepsilon)=+\infty\}.
	\]
	The \textit{upper Bowen metric mean dimension of $f$ with potential $\psi$ on $Z$} is then defined as the following limit 
	\begin{equation*}
		\overline{\mathrm{mdim}}_{\mathrm{M}}^B\,\Big(Z,f,\psi,d\Big)=\limsup_{\varepsilon\to 0}\frac{m\,\Big(Z,f,\psi,\varepsilon\Big)}{|\log \varepsilon|}.
	\end{equation*}

	We also consider the quantities
	\[
	P_n(d, f, \psi , \varepsilon, Z)=\sup\left\{\sum_{x \in E} \exp \left(|\log \varepsilon| S_n \psi (x)\right):\; E \text{ is an }(n, \varepsilon)\text{-separated subset of }Z\right\},
	\]
	and
	\[
	P(d, f, \psi , \varepsilon, Z)=\limsup _{n \rightarrow \infty} \frac{1}{n} \log P_n(d, f, \psi , \varepsilon, Z) .
	\]
 Then, the \emph{upper capacity metric mean dimension of $Z$ with potential $\psi$} is defined by 
	\[\mathrm{\overline{mdim}_M}\,\Big(Z,f,d,\psi \Big)=\limsup_{\varepsilon \rightarrow 0}\frac{1}{|\log\varepsilon|}	P(d, f, \psi , \varepsilon, Z).\]
	
	Recall that the \emph{topological pressure} of $f$ with respect to the potential $\psi $ on $Z \subset X$ is defined by
	\[
	CP(f, \psi , Z)=\lim _{\varepsilon \rightarrow 0} P(d, f,\psi , \varepsilon, Z).
	\]
	Moreover, whenever $Z=X$, the quantity $CP(f,\psi):=CP(f,\psi,X)$ is simply the topological pressure of $f$ with respect to $\psi $. Furthermore, whenever $\psi  \equiv 0$, $CP(f, 0)$ reduces to the \emph{topological entropy of $f$}. 
	
	Alternatively, the upper capacity metric mean dimension of $Z$ with potential $\psi$ may be defined as follows. Let us consider
	\[
	Q(d, f, \psi , \varepsilon, Z)=\limsup _{n \rightarrow \infty} \frac{1}{n} \log Q_n(d,f, \psi , \varepsilon, Z)
	\]
	where
	\[
	Q_n(d, f, \psi, \varepsilon, Z)=\inf \left\{\sum_{y \in F} \exp \left(|\log \varepsilon| S_n\psi(y)\right):\;
	F\text{ is an }(n, \varepsilon)\text{-spanning subset of }Z\right\}.\]
	Then it follows by \cite[Proposition 2.6]{Wurelative} that 
	\[
	\mathrm{\overline{mdim}_M}\,\Big(Z,f,d,\psi \Big)=\limsup_{\varepsilon \rightarrow 0}\frac{1}{|\log\varepsilon|}	Q(d, f, \psi , \varepsilon, Z).
	\]
	By \cite{Chen} we have that 
	\begin{equation}
		\label{eq:ineq_upper_topo_capa_mdim}\mathrm{\overline{mdim}^B_M}\,\Big(Z,f,d,\psi \Big)\leq\mathrm{\overline{mdim}_M}\,\Big(Z,f,d,\psi \Big).
	\end{equation}
	In the case when $Z$ is compact and $f$-invariant we have the equality.
	
	\subsection{Level sets of a continuous map} Given $\varphi \in C(X,\mathbb{R})$, for $\alpha\in\mathbb R$, let us consider the \emph{level set}
	\begin{align}\label{def-target-set}
		K_\alpha=\left\{x\in X: \lim_{n\to\infty}\frac{1}{n}\sum_{j=0}^{n-1}\varphi(f^{j}(x))=\alpha\right\}.
	\end{align}
	Associated to this level set we also consider
	\begin{displaymath}
		\mathcal L_\varphi=\{\alpha\in\mathbb R: K_\alpha\not=\emptyset\}.
	\end{displaymath}
	It is easy to see that $\mathcal L_\varphi $ is a bounded and non-empty set \cite[Lemma 2.1]{TV}. Moreover, if $f$ satisfies the so called specification property (see Section \ref{sec: specification}) then $\mathcal L_\varphi$ is an interval of $\mathbb R$ and, moreover, $\mathcal L_\varphi=\{\int \varphi d\mu :\mu \in \mathcal{M}_f(X)\}$ where $\mathcal{M}_f(X)$ stands for the set of all $f$-invariant probability measures (see \cite[Lemma 2.5]{Tho}).

	\subsection{A measure-theoretic upper metric mean dimension $H_\delta^K(\mu)$} In order to define a measure-theoretic notion of upper metric mean dimension, we follow the approach in \cite{CPV}.
	Let $\mathcal M^{erg}_f(X)$ denote the set of all ergodic elements in $\mathcal M_f(X)$. Given $\mu\in \mathcal M^{erg}_f(X)$, $\varepsilon>0$ and $\delta\in (0,1)$, let us denote by $N_\mu(\varepsilon,\delta,n)$ the minimal number $(n,\varepsilon)$-balls needed to cover a set of $\mu$-measure bigger than $1-\delta$. That is, 
	\[N_\mu(\varepsilon,\delta,n)=\inf_{A\in \mathcal{B}}\{Q_n(d,f,0,\varepsilon,A):\mu(A)>1-\delta\}\]
	where $\mathcal{B}$ denotes the Borel $\sigma$-algebra of $(X,d)$. Then, we define
	\[h_\mu^K(\varepsilon,\delta)=\limsup_{n\to\infty}\frac{1}{n}\log N_\mu(\varepsilon,\delta,n)\]
	where the superscript ``K'' in $h_\mu^K(\varepsilon,\delta)$ stands for ``Katok'' since this quantity comes from the description of the
	metric entropy given by Katok (see \cite{Kat, Walters}).
	The previous notion can be extended to non-ergodic probability measures in $\mathcal{M}_f(X)$ via integration: given $\mu\in\mathcal{M}_f(X)$, define
	\begin{align}\label{convexity-property}
		h_\mu^K(\varepsilon,\delta)=\int_{\mathcal M^{erg}_f(X)} h_m^K(\varepsilon,\delta)\ d\mathbb P_\mu(m),
	\end{align}
	where $\mu=\int_{\mathcal M^{erg}_f(X)} m\ d\mathbb P_\mu(m)$ is the ergodic decomposition of $\mu$.  Observe that, by the definition,  the
	map $m\mapsto h_m^K(\varepsilon,\delta)$ is measurable and integrable. Consequently, the function
	\begin{align*}
		\mu\in \mathcal M_f(X)\mapsto  h_\mu^K(\varepsilon,\delta)
	\end{align*}
	is also affine.
	
	Now, given $\mu\in \mathcal M_f(X)$ and $\delta\in (0,1)$, we define the map $ H_\delta^K\colon \mathcal M_f(X) \to \mathbb R$ by
	\begin{equation}\label{eq: def metric-metric-mean-dim}
		H_\delta^K(\mu)= \sup_{(\mu_\varepsilon)_\varepsilon\in\mathcal{M}(\mu)}\limsup_{\varepsilon\to0}\frac{ h_{\mu_\varepsilon}^K(\varepsilon,\delta)}{|\log\varepsilon|}
	\end{equation}
	where  $\mathcal{M}(\mu)$ stands for the space of sequences of probability measures in $\mathcal M_f(X)$ which converge to $\mu$ in the weak$^{\ast}$-topology. This quantity was introduced in \cite{CPV}, where in Theorem C, they have proved that 
	\begin{align*}
		\mathrm{\overline{mdim}_M}\,\Big(X,f,d,\psi \Big)=\sup\left\{H_\delta^K(\mu)+\int_X \psi\ d\mu:\mu\in \mathcal M^{erg}_f(X)\right\}.
	\end{align*}

	\begin{remark}\label{rmk_important}
		It is important to notice that $h_\mu^K(\varepsilon,\delta)$ can be defined in terms of $(n,\varepsilon)$-spanning sets. More precisely, 
		\[h_\mu^K(\varepsilon,\delta)=\limsup_{n\to\infty}\frac{1}{n}\log b_\mu(\varepsilon,\delta,n),\]
		where $b_\mu(\varepsilon,\delta,n)$ denotes the minimal cardinality of a $(n,\varepsilon)$-spanning set contained in a set of $\mu$-measure bigger than $1-\delta$. 
	\end{remark}
	
	\subsection{The measure-theoretic quantities $\mathrm{H_\varphi\overline{mdim}_M}\,(f,\alpha,d,\psi)$ and $H(\mu)$} Given $\varphi\in C(X,\mathbb R)$ and $\alpha \in\mathbb R$, let us consider
	\begin{equation*}
		\mathcal M_{f}(X,\varphi,\alpha)=\left\{\mu\in\mathcal{M}_f(X):\int \varphi\;d\mu=\alpha\right\}.
	\end{equation*}
	A simple observation is that $\mathcal M_{f}(X,\varphi,\alpha)\neq \emptyset$ for every $\alpha\in \mathcal{L}_\varphi$ (see \cite[Lemma 4.1]{TV}).
	
	Let $\mu\in\mathcal M_f(X)$. We say that $\xi=\{C_1,\ldots,C_k\}$ is a measurable partition of $X$ if every $C_i$ is a measurable set, $\mu\left(X\setminus\cup_{i=1}^kC_i\right)=0$ and $\mu\left(C_i\cap C_j\right)=0$ for every $i\neq j$. The \emph{entropy} of $\xi$ with respect to $\mu$ is given by
	\[
	H_\mu(\xi)=-\sum_{i=1}^{k}\mu(C_i)\log(\mu(C_i)).
	\]
	Given a measurable partition $\xi$, we consider $\xi^n=\bigvee_{j=0}^{n-1}f^{-j}\mathcal \xi$. Then, the \emph{metric entropy of $(f,\mu)$ with respect to $\xi$} is given by
	\[
	h_\mu(f,\xi)=\lim_{n\to +\infty}\frac{1}{n} H_\mu(\xi^n).
	\] 
	Using this quantity we define 
	\begin{align}\label{eq: def Hmdim} 
		&\mathrm{H_\varphi\overline{mdim}_M}\,(f,\alpha,d,\psi)\\
		\nonumber &=\limsup_{\varepsilon \rightarrow 0}\frac{1}{|\log\varepsilon|}\sup_{\mu\in \mathcal M_{f}(X,\varphi,\alpha)}\inf_{|\xi|<\varepsilon}\left(h_\mu(f, \xi)+ \int \psi|\log\varepsilon| d \mu\right)
	\end{align}
	where $|\xi|$ denotes the diameter of the partition $\xi$ and the infimum is taken over all finite measurable partitions of $X$ satisfying $|\xi|<\varepsilon$.

	Moreover, we may also define the quantity
	\begin{align}
		\label{def:metric;metric}
		H(\mu)= \sup_{(\mu_\varepsilon)_\varepsilon\in\mathcal{M}(\mu)}\limsup_{\varepsilon\to0}\frac{ \inf_{|\xi|<\varepsilon}h_{\mu_\varepsilon}(f,\xi)}{|\log\varepsilon|}.
	\end{align}

	Finally, we recall that the \emph{metric entropy of $(f,\mu)$} is given by
	\[
	h_\mu(f)=\sup_\xi h_\mu(f,\xi)
	\]
	where the supremum is taken over all finite measurable partitions $\xi$ of $X$.

	\subsection{Specification property} \label{sec: specification}
	We say that $f$ satisfies the \emph{specification property} if for every $\epsilon > 0$, there exists an integer $m = m(\epsilon )$ such that for any collection of finite intervals $ I_j = [a_j, b_j ] \subset \mathbb{N}$, $j = 1, \ldots, k $, satisfying $a_{j+1} - b_j \geq m(\epsilon )$ for every $j = 1, \ldots, k-1 $ and any $x_1, \ldots, x_k$ in $X$, there exists a point $x \in X$ such that
	\begin{equation*}
		d(f^{p + a_j}x, f^p x_j) < \epsilon 
	\end{equation*}
	for all $ p = 0, \ldots, b_j - a_j$  and every $ j = 1, \ldots, k.$
	The specification property is present in many interesting examples. For instance, every topologically mixing locally maximal hyperbolic set has the specification property and factors of systems with specification have specification (see for instance \cite{KH}). Other examples of systems satisfying this property that have positive metric mean dimension may be found in \cite[Section IV]{BR}.

	\subsection{Main result}
	The main result of this paper is the following.
	\begin{theorem}\label{thm:main1}
		Suppose $f:X \to X$ is a continuous map with the specification
		property. Let $\varphi \in C(X,\mathbb R)$ and $\alpha\in\mathbb R$ be such that $K_\alpha\not=\emptyset$.
		Then, given $\psi \in C(X,\mathbb R)$, for every $\delta \in (0,1)$ we have that
		\[
		\begin{split}
			\mathrm{\overline{mdim}_M}\,\Big(K_\alpha,f,d,\psi \Big)&=	\mathrm{\overline{mdim}_M^B}\,\Big(K_\alpha,f,d,\psi \Big)\\
			&=\mathrm{H_\varphi\overline{mdim}_M}\,(f,\alpha,d,\psi)\\
			&=\sup\left\{ H(\mu)+\int_X \psi\  d \mu :\mu\in \mathcal M_f(X,\varphi , \alpha)\right\}\\
			&=\sup\left\{H_\delta^K(\mu)+\int_X \psi\ d\mu:\mu\in \mathcal M_f(X,\varphi , \alpha)\right\} .
		\end{split}
		\]
	\end{theorem}

	The paper is organized as follows. In Section 3, we prove Theorem \ref{thm:main1}. More specifically, Sections 3.1 and 3.2 are dedicated to proving the first two equations. Section 3.3 focuses on proving the last equation. In Section 4, we apply these results in the context of suspension flows.

	\section{Proof of Theorem \ref{thm:main1}}
	
	In this section we present the proof of Theorem \ref{thm:main1} which is based on ideas developed in \cite{BR,TV, Tho,MR2765447}. In order to simplify notation, we will denote the function $|\log\varepsilon|\psi $ simply by $\psi_\varepsilon$ for any $\varepsilon>0$ and $\psi\in C(X,\mathbb{R})$. 
	
	\subsection{Upper bounds for $\mathrm{\overline{mdim}_M^B}\,\Big(K_\alpha,f,d,\psi \Big)$}
	We start with an auxiliary lemma which is a simple adaptation of \cite[Lemma 3.2]{Tho} to our setting. We include its proof for the sake of completeness. Moreover, it helps to clarify the role of the specification property in the proof of Theorem \ref{thm:main1}.
	
	\begin{lemma}\label{lem1}
		Suppose that $f\colon X\to X$ has the specification property. Then, given $\gamma>0$, for any $\varepsilon>0$, there exist $Z \subset K_\alpha$, $t_k \rightarrow \infty$ and $\varepsilon_k \rightarrow 0$ such that, if $p \in Z$, then
		\begin{equation}\label{eq3}
			\left|\frac{1}{m} S_m \varphi(p)-\alpha\right| \leqslant \varepsilon_k \quad \text { for all } m \geqslant t_k
		\end{equation}
		and, moreover,
		\[	Q(d, f, \psi_{\varepsilon} , \varepsilon, Z) \ge Q(d, f, \psi_{2\varepsilon} , 2\varepsilon, K_\alpha)-4 \gamma.\]
	\end{lemma}
	\begin{proof}
		Fix $\varepsilon>0$. Given $\delta>0$, let us consider the set
		\[
		K(\alpha, n, \delta)=\left\{x \in K_\alpha:\left|\frac{1}{m} S_m \varphi(x)-\alpha\right| \leqslant \delta \text { for all } m \geqslant n\right\} .
		\]
		Then, we have that $K_\alpha=\bigcup_{n=1}^\infty K(\alpha, n, \delta)$ and $K(\alpha, n, \delta) \subset K(\alpha, n+1, \delta)$. Consequently, 
		\[Q(d, f, \psi_{2\varepsilon} , 2\varepsilon, K_\alpha)=\lim _{n \to \infty} Q(d, f, \psi_{\varepsilon} , 2\varepsilon, K(\alpha, n, \delta)).\]
		Fix an arbitrary sequence $\delta_k \rightarrow 0$. Then, for each $k\in \N$ there exists $M_k \in \mathbb{N}$ so that
		\[
		Q(d, f, \psi_{2\varepsilon} , 2\varepsilon, K\left(\alpha, M_k, \delta_k\right))\ge Q(d, f, \psi_{2\varepsilon} , 2\varepsilon, K_\alpha)-\gamma .
		\]
		
		Let us write $K_k:=K\left(\alpha, M_k, \delta_k\right)$ for each $k\in\mathbb{N}$ and let $m_k=m\left(\varepsilon / 2^k\right)$ be given by the specification property associated to $\varepsilon/2^k$. Moreover, take a sequence of natural numbers $N_k \rightarrow \infty$ increasing sufficiently fast so that
		\begin{align}
			N_{k+1}>\max \left\{e^{\sum_{i=1}^k\left(N_i+m_i\right)}, e^ {M_{k+1}}, e^ {m_{k+1}}\right\} \label{eq1}
		\end{align}
		and
		\begin{align*}
			Q_{N_k}(d, f, \psi_{2\varepsilon}, 2\varepsilon, K_k)>e^{N_k\left(Q(d, f, \psi_{2\varepsilon} , 2\varepsilon, K_\alpha)-3 \gamma\right)} .
		\end{align*}

		Let $(t_k)_{k\in \N}$ be the sequence defined recursively as $t_1=N_1$ and $t_k=t_{k-1}+m_k+N_k$ for $k \geqslant 2$. Then, by \eqref{eq1}, we have that $t_k / N_k \rightarrow 1$ and $t_{k-1} / t_k \rightarrow 0$, as $k\to\infty$.
		Our objective now is to construct the set $Z$ with the desired properties. For this purpose, for each $i\in\mathbb{N}$, fix $x_i \in K_i$. Then, using the specification property we can build a sequence of points $z_1, z_2, \ldots, z_k, \ldots$ in the following way: let $z_1=x_1$ and choose $z_2$ to satisfy
		\[
		d_{N_1}\left(z_2, z_1\right)<\varepsilon / 4 \quad \text{ and } \quad d_{N_2}\left(f^{N_1+m_2}(z_2), x_2\right)<\varepsilon / 4.
		\]
		For $k\geq 3$ take $z_k$ satisfying
		\[
		d_{t_{k-1}}\left(z_{k-1}, z_k\right)<\varepsilon / 2^{k} \quad \text { and } \quad d_{N_k}\left(f^{t_{k-1}+m_k} (z_k), x_k\right)<\varepsilon / 2^{k} .
		\]
		
		From the properties of the sequence $(z_k)_{k\in \N}$ we have that $\overline{B}_{t_{k+1}} \left(z_{k+1}, \varepsilon / 2^{k}\right)\subset \overline{B}_{t_k}\left(z_k, \varepsilon / 2^{k-1}\right)$ and, consequently, the point $p:=\bigcap_{k=1}^{\infty} \overline{B}_{t_k}\left(z_k, \varepsilon / 2^{k+1}\right)$ is well defined. We define $Z$ to be the set of all points $p$ constructed in this way and show that it has the desired properties.
		
		We start observing that for any $p \in Z$, there exists $x_k \in K_k$ so that $$d_{N_k}\left(f^{t_{k-1}+m_{k}} (p), x_k\right)<\varepsilon / 2^{k-1}.$$ Consequently,
		\[
		S_{t_k} \varphi(p) \leqslant S_{N_k} \varphi\left(x_k\right)+N_k \operatorname{Var}\left(\varphi, \varepsilon / 2^{k-1}\right)+(t_{k-1}+m_{k})\|\varphi\|
		\]
		where 
		\[\operatorname{Var}(\varphi,\varepsilon)=\sup \{ |\varphi(x)-\varphi(y)|: \; d(x,y)<\varepsilon \} \text{ and } \|\varphi \|=\sup_{x\in X}|\varphi(x)|.\] 
		Therefore, we can find a sequence $\varepsilon_k^{\prime} \rightarrow 0$ such that, for any $p \in Z$, we have
		\[
		\left|\frac{1}{t_k} S_{t_k} \varphi(p)-\alpha\right|<\varepsilon_k^{\prime} .
		\]

		Let us now consider $n\in \N$ such that $t_k< n<t_{k+1}$. Then, if $n-t_k+m_{k+1} \ge M_{k+1}$, there exists $x_{k+1} \in K_{k+1}$ such that $d_{N_{k+1}}\left(f^{t_k+m_{k+1}}(p), x_{k+1}\right)<\varepsilon / 2^{k+1}$, and thus 
		\[
		S_n \varphi(p) \leqslant t_k\left(\alpha+\varepsilon_k^{\prime}\right)+\left(n-t_k\right)\left(\alpha+\delta_{k+1}+\operatorname{Var}\left(\varphi, \varepsilon / 2^{k+1}\right)\right)+m_{k+1}\|\varphi\| .
		\]
		Now, if $n-t_k \leqslant M_{k+1}$, then
		\begin{align*}
			\frac{1}{n} S_n \varphi(p) \le \frac{t_k}{n}\left(\alpha+\varepsilon_k^{\prime}\right)+\frac{n-t_k}{n}\|\varphi\| \le \alpha+\varepsilon_k^{\prime}+\frac{M_{k+1}}{N_k}\|\varphi\| .
		\end{align*}
		Consequently, taking 
		\[\varepsilon_k=\max \left\{\varepsilon_k^{\prime}, \delta_{k+1}+\operatorname{Var}\left(\varphi, \varepsilon / 2^{k+1}\right)\right\}+\max \left\{M_{k+1} / N_k, m_{k+1} / N_k\right\}\|\varphi\|\]
		it follows that \eqref{eq3} holds. 
		
		Let us now show that the second claim in the lemma also holds. Take a $\left(t_k, \varepsilon\right)$-spanning set $F_k$ of $Z$ satisfying 
		\begin{equation}\label{eq: l1 aux}
			\sum_{x \in F_k} \exp \left( S_{t_k} \psi_{\varepsilon}(x)\right) =Q_{t_k}(d, f, \psi_{\varepsilon}, \varepsilon, Z).
		\end{equation}
		Then, $f^{t_{k-1}+m_k} (F_k)$ is a $\left(N_k, \varepsilon\right)$-spanning set for $f^{t_{k-1}+m_k} (Z)$. Since for any $x\in K_k$, there exists $z\in f^{t_{k-1}+m}(Z)$ such that $d_{N_k}(x,z)<\varepsilon/2^k$, it follows that $f^{t_{k-1}+m_k} (F_k)$ is an $\left(N_k, 2 \varepsilon\right)$-spanning set for $K_k$. Consequently,
		\begin{align*}
			\sum_{x \in F_k} \exp \left( S_{N_k} \psi_{2\varepsilon}(f^{t_{k-1}+m_k}( x)) \right) & \ge	Q_{N_k}(d, f, \psi_{2\varepsilon} , 2\varepsilon, K_k)\\
			&>\exp \left( N_k(Q(d, f, \psi_{2\varepsilon} , 2\varepsilon, K_\alpha)-3 \gamma)\right).
		\end{align*}
		Thus, for $k$ large enough,
		\begin{align}
			\sum_{x \in F_k} \exp \left( S_{t_k} \psi_{2\varepsilon}(x)\right) & \ge \exp\left( N_k(Q(d, f,\psi_{2\varepsilon} , 2\varepsilon, K_\alpha)-3 \gamma)+\left(t_{k-1}+m_k\right) \inf \psi\right)  \nonumber \\
			& \geqslant \exp \left(t_k(Q(d, f, \psi_{2\varepsilon} , 2\varepsilon, K_\alpha)-4 \gamma)\right) . \label{eq: l1 final}
		\end{align}
		Finally, taking the ``$\liminf$'' of the sequence $t_k^{-1} \log \left( Q_{t_k}(d, f, \psi_ {\varepsilon} , \varepsilon, Z) \right)$ it follows from \eqref{eq: l1 aux} and \eqref{eq: l1 final} that
		\[
		Q(d, f, \psi_{\varepsilon}  ,  \varepsilon, Z)>Q(d, f, \psi_{2\varepsilon} , 2\varepsilon, K_\alpha)-4 \gamma
		\]
		as claimed.
	\end{proof}

 We are now in position to provide the upper bound for $\mathrm{\overline{mdim}_M}\,\Big(K_\alpha,f,d,\psi \Big)$, which, by \eqref{eq:ineq_upper_topo_capa_mdim}, gives an upper bound for $\mathrm{\overline{mdim}_M^B}\,\Big(K_\alpha,f,d,\psi \Big)$. The idea is to construct some special measures that live on some $(n,\varepsilon)$-separated sets of the set $Z$ given by Lemma \ref{lem1} for which the measure-theoretic metric mean dimension is ``high''. We follow the approach of \cite[Theorem 9.10]{Walters} and \cite{Tho}.
	
	Given any $\varepsilon>0$, let $E_n$ be an $(n, \varepsilon)$-separated set of $Z$ satisfying
	\[
	\sum_{x \in E_n} \exp \left( S_n \psi_{\varepsilon}(x)\right)=	P_n(d, f, \psi_{\varepsilon} , \varepsilon, Z)
	\]
	and set $P_{n}:=P_n(d, f, \psi_{\varepsilon} , \varepsilon, Z)$. Let $\sigma_{n}$ be the probability measure on $X$ given by
	\[
	\sigma_n=\frac{1}{P_n} \sum_{x \in E_n} \exp \left( S_n \psi_{\varepsilon}(x)\right) \delta_x 
	\]
	and consider
	\[
	\mu_n=\frac{1}{n} \sum_{i=0}^{n-1} \sigma_n \circ f^{-i} .
	\]
	Fix a subsequence $\{n_j\}_{j=1}^\infty$ of $\mathbb{N}$ such that 
	\[P(d, f, \psi_{\varepsilon} , \varepsilon, Z) =\lim  _{j \rightarrow \infty} \frac{1}{n_j} \log P_{n_j}.\]
	Then, passing to a subsequence of $\{n_j\}_{j=1}^\infty$, if necessary, we may assume that $\mu_{n_j}$ converges. Denote by $\mu$ be the limiting measure of this sequence. Then it is straightforward to check that $\mu \in \mathcal{M}_f(X)$. Moreover, $\int \varphi d \mu=\alpha$. Indeed, given any $n \in \mathbb{N}$, let $k$ be the unique number such that $t_k \leqslant n<t_{k+1}$. Then, using Lemma \ref{lem1}, we have that
	\begin{align*}
		\int \varphi d \mu_n & =\frac{1}{P_n} \frac{1}{n} \sum_{x \in E_n} S_n \varphi(x) e^{S_n \psi_{\varepsilon}(x)} \\
		& \leqslant \frac{1}{P_n} \frac{1}{n} \sum_{x \in E_n} n\left(\alpha+\varepsilon_k\right) e^{S_n \psi_{\varepsilon}(x)} \\
		& =\alpha+\varepsilon_k.
	\end{align*}
	Similarly, we can prove $	\int \varphi d \mu_n \ge\alpha-\varepsilon_k$. Thus, $\int \varphi d \mu=\alpha$ as claimed and $\mu \in \mathcal M_f(X,\varphi , \alpha)$.

	Let us consider now a finite measurable partition $\xi$ of $X$ such that $|\xi|<\varepsilon$ and $\mu(\partial \xi)=0$ where $\partial \xi$ stands for the boundary of the partition $\xi$ which is just the union of the boundaries of all the elements of the partition. Then by an argument similar to that of \cite[Theorem 9.10]{Walters}, we have
	\[
	H_{\sigma_n}\left(\bigvee_{i=0}^{n-1} f^{-i} \xi\right)+\int S_n  \psi_{\varepsilon} d \sigma_n=\log P_n.
	\]
	Moreover, since $\mu(\partial \xi)=0$, for any $q \in \mathbb{N}$,
	\[
	\lim _{j \rightarrow \infty} H_{\mu_{n_j}}\left(\bigvee_{i=0}^{q-1} f^{-i} \xi\right)=H_\mu\left(\bigvee_{i=0}^{q-1} f^{-i} \xi\right) .
	\]
	Now, for any $q,n\in\mathbb{N}$ with $0<q<n-1$, we have
	\[
	\frac{q}{n} \log P_n \leq H_{\mu_n}\left(\bigvee_{i=0}^{q-1} f^{-i} \xi\right)+q \int \psi_{\varepsilon} d \mu_n+2 \frac{q^2}{n} \log \# \xi 
	\]
	where $\# \xi$ denotes the number of elements of the partition $\xi$ (check the proof of \cite[Theorem 9.10]{Walters} for the details in this calculations).
	Thus, replacing $n$ by $n_j$ and letting $j\to\infty$, we obtain
	\[
	q P(d, f, \psi_{\varepsilon} , \varepsilon, Z) \le  H_\mu\left(\bigvee_{i=0}^{q-1} f^{-i} \xi\right)+q \int \psi_{\varepsilon} d \mu .
	\]
	Finally, dividing everything by $q$ and letting $q \rightarrow \infty$, we get that
	\[
	P(d, f, \psi_{\varepsilon} , \varepsilon, Z) \le h_\mu(f, \xi)+ \int \psi_{\varepsilon} d \mu .
	\]
	Thus, since $\xi$ is an arbitrary finite measurable partition with $\mu(\partial \xi)=0$ and $|\xi|<\varepsilon$, it follows that 
	\[
	P(d, f, \psi_{\varepsilon} , \varepsilon, Z) \le \inf_{\mu(\partial \xi)=0,\; |\xi|<\varepsilon}\left( h_\mu(f, \xi)+ \int \psi_{\varepsilon} d \mu \right).
	\]
Then, using \cite[Lemma 2.2]{Wurelative} and   \cite[Proposition 2.5]{yuan2024variational}, we get that 
		\begin{align*}
			Q(d, f, \psi_{\varepsilon} , \varepsilon, Z) \leq P(d, f, \psi_{\varepsilon} , \varepsilon, Z) \le \inf_{|\xi|<\varepsilon} \left( h_\mu(f, \xi)+ \int \psi_{\varepsilon} d \mu \right).
		\end{align*}
		Consequently, by Lemma \ref{lem1},
		\begin{align}
			\label{eq:fundamental}
			Q(d, f, \psi_{2\varepsilon} , 2 \varepsilon, K_\alpha)-4\gamma\leq \inf_{\mu(\partial \xi)=0,\; |\xi|<\varepsilon}\left( h_\mu(f, \xi)+ \int \psi_{\varepsilon} d \mu \right).
		\end{align}

		Finally, by \eqref{eq:fundamental} we get
		\begin{align}
			\mathrm{\overline{mdim}_M}\,\Big(K_\alpha,f,d,\psi \Big)
			= &\limsup_{\varepsilon \rightarrow 0}\frac{1}{|\log 2\varepsilon|}	Q(d, f, \psi_{2\varepsilon} , 2 \varepsilon, K_\alpha)   \label{eq: lower b1 aux}\\
			\le &\limsup_{\varepsilon \rightarrow 0}\frac{1}{|\log\varepsilon|}\sup_{\nu\in \mathcal M_f(X,\varphi , \alpha)}\inf_{|\xi|<\varepsilon} \left( h_\nu(f, \xi)+ \int \psi_{\varepsilon}  d \nu+4\gamma\right)\nonumber\\
			= &\mathrm{H_\varphi\overline{mdim}_M}\,(f,\alpha,d,\psi)\nonumber
		\end{align}
		which is one of the upper bounds for $\mathrm{\overline{mdim}_M}\,\Big(K_\alpha,f,d,\psi \Big)$ that we were looking for.
		
		For the remaining upper bound we notice that the probability measure $\mu$ constructed as in \eqref{eq:fundamental} depends on $\varepsilon>0$ fixed. In particular, we have a family $(\mu_\varepsilon)_\varepsilon$ which satisfies  
		
		\begin{align}\label{eq:fundamental2}
			\frac{Q(d, f, \psi_{2\varepsilon} , 2 \varepsilon, K_\alpha)-4\gamma}{|\log\varepsilon|} 
			\le\frac{\inf_{|\xi|<\varepsilon} \left( h_{\mu_\varepsilon}(f, \xi)+ \int_X \psi_{\varepsilon}  d \mu_\varepsilon \right)}{|\log\varepsilon|},
		\end{align}
		for every $\varepsilon>0$. Let $\bar\mu\in \mathcal{M}_f(X)$ be an accumulation point of $(\mu_\varepsilon)_\varepsilon$. It is an immediate consequence of the definition of the weak$^\ast$ topology that $\bar\mu \in \mathcal M_f(X,\varphi , \alpha)$.
		Hence, using \eqref{eq:fundamental2}  and recalling the definition of $H(\bar\mu)$ we get that 
		\begin{align*}
			\limsup_{\varepsilon \rightarrow 0}\frac{Q(d, f, \psi_{2\varepsilon} , 2 \varepsilon, K_\alpha)-4\gamma}{|\log\varepsilon|} 
			&\le \limsup_{\varepsilon \rightarrow 0}\frac{\inf_{|\xi|<\varepsilon} \left( h_{\mu_\varepsilon}(f, \xi)+ \int_X \psi_{\varepsilon}  d \mu_\varepsilon \right)}{|\log\varepsilon|}\\
			&\leq \sup_{(\mu_{\theta })_{\theta }\in\mathcal{M}(\bar\mu)}\limsup_{\theta\to0}\frac{ \inf_{|\xi|<\theta}h_{\mu_{\theta }}(f,\xi)}{|\log\theta|} +\int_X \psi  d \bar\mu    \\
			&=H(\bar\mu)+\int_X \psi\  d \bar\mu .
		\end{align*}
		Thus, using by \eqref{eq: lower b1 aux} and \eqref{eq:ineq_upper_topo_capa_mdim} we conclude that 
		\[\mathrm{\overline{mdim}_M^B}\,\Big(K_\alpha,f,d,\psi \Big)\leq H(\bar\mu)+\int_X \psi\  d \bar\mu \]
		which is the second upper bound for $\mathrm{\overline{mdim}_M^B}\,\Big(K_\alpha,f,d,\psi \Big)$ that we were trying to establish. The third one will be obtained in Section \ref{sec: conclusion of proof}.

	\subsection{Lower bounds for $\mathrm{\overline{mdim}_M^B}\,\Big(K_\alpha,f,d,\psi \Big)$}
	The strategy to get the lower bounds for $\mathrm{\overline{mdim}_M^B}\,\Big(K_\alpha,f,d,\psi \Big)$ consists in constructing a fractal set $F$ contained in $K_\alpha$ and a special probability measure $\eta$ supported on $F$ that satisfies the hypothesis of the so called Pressure Distribution Principle (see Lemma \ref{lem:pressure distribution}). This will be enough
	to get the desired inequality. As a step towards the definition of $F$, we introduce three families of finite sets $S_k$, $\mathcal{C}_k$ and $\mathcal T_k$ which play a major role in the construction of $F$ and $\eta$. This idea was already explored in \cite{BR,Liu-Liu,TV,Tho,MR2765447} and we follow them closely.
	
	Fix $\gamma>0$ and let $\{\delta_k\}_{k\in\mathbb{N}}$ be a decreasing sequence converging to $0$. Take $\varepsilon
	=\varepsilon(\gamma)>0$ sufficiently small and $\mu\in \mathcal M_{f}(X,\varphi,\alpha)$ so that 
	\begin{align}\label{eq: choice vareps H}
		\frac{\inf_{|\xi|<5\varepsilon}\left( h_{\mu}(f,\xi)+ \int\psi_{5\varepsilon}  d\mu \right) -5\gamma }{|\log5\varepsilon|} \geq\mathrm {H_\varphi\overline{mdim}_M}\,(f,\alpha,d,\psi)-\gamma,
	\end{align}
	\begin{equation}\label{eq: choice varepsilon for P}
		\frac{m(K_\alpha,f,\psi_{\varepsilon/2},\varepsilon/2)}{|\log \varepsilon/2|}\leq \mathrm{\overline{mdim}_M^B}\,(K_\alpha,f,d,\psi) +\gamma
	\end{equation}
	and 
	\begin{equation}\label{eq: coiche vareps gamma}
		\frac{|\log 5\varepsilon|}{|\log \varepsilon/2|}\geq (1-\gamma) \; \text{ and }\; \left| \frac{\log 4 \int\psi d\mu-2  (\log5\varepsilon/4)\operatorname{Var}(\psi, 2\varepsilon )}{\log \varepsilon/2}\right|<\gamma.
	\end{equation}

	Let $\mathcal U$ be a finite open cover of $X$ with diameter $\mathrm{diam}(\mathcal U)\leq 5\varepsilon$ and Lebesgue number $\mathrm{Leb}(\mathcal U)\geq \frac{5\varepsilon}{4}$. In the next lemma we present an auxiliary measure which is a finite combination of ergodic measures and ``approximates" $\mu$. Its proof may be obtained by making trivial adjustments to the proof of \cite[Lemma 7]{BR}. Before we proceed with the statement, let us recall that given a partition $\xi$ of $X$, $\xi\succ\mathcal U$ means that $\xi$ refines $\mathcal{U}$, that is, each element of $\xi$ is contained in an element of $\mathcal U$.

	\begin{lemma}\label{lem: approx}
		For each $k\in\mathbb N$, there exists a measure $\nu_k \in \mathcal M_{f}(X)$ satisfying    
		\begin{itemize}
			\item[(a)] $\nu_k=\displaystyle\sum_{i=1}^{j(k)} \lambda_i\nu^k_i$,  where $\lambda_i>0$, $\displaystyle\sum_{i=1}^{j(k)} \lambda_i=1$ and $\nu^k_i\in  \mathcal M_f^{\text{erg}}(X) $;
			\item[(b)] $ \displaystyle \inf_{\xi\succ\mathcal U} h_\mu\left(f,\xi\right)\leq \inf_{\xi\succ\mathcal U}\displaystyle h_{\nu_k}\left(f,\xi\right)+\delta_k/2$;
			\item[(c)] $\left|\displaystyle\int_X \varphi\; d\nu_k-\displaystyle\int_X\varphi\; d\mu\right|<\delta_k$;
			\item[(d)] $\left|\displaystyle\int_X \psi\; d\nu_k-\displaystyle\int_X\psi\; d\mu\right|<\delta_k$.
		\end{itemize}
	\end{lemma}
	
	Let $\nu_k$ be as in the previous lemma. Using the fact that each measure $\nu^k_i$ is ergodic, by the proof of \cite[Theorem 9]{Shi} there exists a finite Borel measurable partition $\xi_k$ which refines $\mathcal U$ so that
	\begin{align}\label{eq: relation h part x sep}
		h_{\nu^k_i}^K(5\varepsilon,\gamma)\leq h_{\nu_i^k}(f,\xi_k)\leq h_{\nu^k_i}^K(5\varepsilon/4,\gamma)+\delta_k.
	\end{align}
	Now, take a finite Borel partition $\xi$ refining $\mathcal{U}$ with $\mu(\partial \xi)=0$ such that
	\begin{displaymath}
		h_\mu\left(f,\xi\right)-\delta_k/2-\gamma/2\leq \inf_{\zeta \succ\mathcal U}\displaystyle h_{\nu_k}\left(f,\zeta \right).
	\end{displaymath}
	In particular, since $\xi_k \succ \mathcal{U}$,
	\begin{equation}\label{eq: hmu X hnu}
		h_\mu\left(f,\xi\right)-\delta_k/2-\gamma/2\leq  h_{\nu_k}\left(f,\xi_k\right).
	\end{equation}
	Moreover, since $\xi\succ\mathcal U$, it follows that $|\xi|<5\varepsilon$. Thus, by \eqref{eq: choice vareps H},
	\begin{align}\label{eq:escolha de xi}
		\frac{h_{\mu}(f,\xi)+\int \psi_{5\varepsilon}  d\mu-5\gamma}{|\log5\varepsilon|} \geq\mathrm {H_\varphi\overline{mdim}_M}\,(f,\alpha,d,\psi)-\gamma.
	\end{align}

	\subsubsection{Construction of $\mathcal{S}_k$} We start observing that,
	since each $\nu^k_i$ is ergodic, there exists $\ell_k\in \mathbb N$ large enough for which the set 
	\begin{align}
		\label{def:sets}
		Y_{i}(k)=\left\{x\in X: \left|\frac{1}{n}\sum_{j=0}^{n-1}\varphi(f^j(x))-\int_X\varphi\; d\nu^k_i\right|<\delta_k\;\; \forall \; n\geq \ell_k\right\}
	\end{align}
	has $\nu^k_i$-measure bigger than $1-\gamma$ for every $k\in\mathbb  N$ and $i\in\{1,\dots,j(k)\}$. 
	By \cite[Lemma 3.6]{Tho}, there exists $\hat n_k\uparrow\infty$ with $[\lambda_i\hat n_k]\geq \ell_k$ and a $([\lambda_i\hat n_k],5\varepsilon/4)$-separated set $E_{k,i}([\lambda_i\hat n_k],5\varepsilon/4)$ of $Y_i(k)$ with maximal cardinality such that 
	\[M_{k, i}:=\sum_{x \in E_{k,i}([\lambda_i\hat n_k],5\varepsilon/4)} \exp \left( \sum_{i=0}^{\hat n_k-1} \psi_{5\varepsilon/4} \left(f^i (x)\right)\right)\]
	satisfies
	\begin{equation}\label{eq: Mki inequality}
		M_{k, i} \ge  \exp \left(\left[\lambda_i \hat{n}_k\right]\left(h^K_{\nu_i^k}(5\varepsilon/4,\gamma)+\int  \psi_{5\varepsilon/4}  d \nu_i^k-\frac{4}{j(k)} \gamma \right)\right) .
	\end{equation}
	Furthermore,  the sequence $\hat n_k$ can be chosen such that $\hat n_k\geq 2^{ m_k({j(k)}-1)}$ where $m_k=m(\varepsilon/2^{k})$ is as in the definition of the specification property. Let $n_k:=m_k({j(k)}-1)+\sum_{i}[\lambda_i\hat n_k]$. Observe that $n_k/\hat n_k\to 1$.
	
	By the specification property, for each
	$$(x_1,\ldots,x_{j(k)})\in\prod_{i=1}^{j(k)}E_{k,i}([\lambda_i\hat n_k],5\varepsilon/4),$$
	there exists $y=y(x_1,\dots,x_{j(k)})\in X$ so that the pieces of orbits
	\begin{displaymath}
		\{x_i,f(x_i),\dots,f^{[\lambda_i\hat n_k]-1}(x_i): i= 1,\ldots,j(k)\}
	\end{displaymath}
	are $\varepsilon/2^k$-shadowed by $y$ with gap $m_k$. It is then proved in \cite[Proposition 6]{BR} that if $(x_1,\dots,x_{j(k)})\not=(x'_1,\dots,x'_{j(k)})$ then $y(x_1,\dots,x_{j(k)})\not=y(x'_1,\dots,x'_{j(k)})$. Moreover, it is observed that 
	\begin{align*}
		\mathcal S_k=\{y(x_1,\dots,x_{j(k)}): x_i\in E_{i,k}([\lambda_i\hat n_k],5\varepsilon/4) \text{ for } i=1,\ldots,j(k)\}
	\end{align*}
	is a $(n_k, 9\varepsilon/8)$-separated set with cardinality $M_k:=\prod_{i=1}^{j(k)}M_{k,i}$. Combining \eqref{eq: relation h part x sep}, \eqref{eq: hmu X hnu} and \eqref{eq: Mki inequality} with the the choices of $\varepsilon$, $\gamma$ and $n_k$ and recalling that $n_k/\hat n_k\to 1$ we get that  for $k$ sufficiently large
	\begin{align}\label{eq:200} 
		M_k&\geq \exp{\left(\sum_{i=1}^{j(k)}[\lambda_i\hat n_k]\left(h_{\nu^k_i}^K(5\varepsilon/4,\gamma)+\int \psi_{5\varepsilon/4} d\nu_i^k-\frac{4\gamma}{j(k)}\right)\right)}\\  \nonumber
		&\geq \exp{\left((1-\gamma)\hat n_k\sum_{i=1}^{j(k)}\lambda_i\left(h_{\nu^k_i}^K(5\varepsilon/4,\gamma)+\int \psi_{5\varepsilon/4} d\nu_i^k\right)-4(1-\gamma)\hat n_k\gamma\right)}\\  \nonumber
		&\geq \exp{\left((1-\gamma)\hat n_k \left(h_{\nu_k}(f,\xi_k)+ \int \psi_{5\varepsilon/4} d\nu_k-4 \gamma-\delta_k\right)\right)}\\  \nonumber
		&\geq \exp{\left( (1-\gamma)^2n_k\left(h_{\mu}(f,\xi)+\int \psi_{5\varepsilon/4} d\mu-9\gamma/2 - 3\delta_k \right)\right)}\\ \nonumber
		&\geq \exp{\left( (1-\gamma)^2n_k\left(h_{\mu}(f,\xi)+\int \psi_{5\varepsilon/4} d\mu-5\gamma \right)\right)}.
	\end{align}
	Moreover, given $y=y(x_1,\dots,x_{j(k)}) \in \mathcal S_k$, we have from the proof of \cite[Proposition 6]{BR} that for sufficiently large $k$,
	\begin{align}\label{eq:201}
		\left|\frac{1}{n_k}S_{n_k}\varphi(y)-\alpha\right|\leq 2\delta_k+\mathrm{Var}(\varphi,\varepsilon/2^k)+\frac{1}{k}.
	\end{align}

	\subsubsection{Construction of $\mathcal{C}_k$} Now we begin to construct the intermediate sets $\left\{\mathcal{C}_k\right\}_{k \in \mathbb{N}}$. Let $N_k$ be a sequence in $\N$ that increases to $\infty$ sufficiently fast so that
	\[
	\lim _{k \rightarrow \infty} \frac{n_{k+1}+m_{k+1}}{N_k}=0 \;\text{ and } \; \lim _{k \rightarrow \infty} \frac{N_1\left(n_1+m_1\right)+\ldots+N_k\left(n_k+m_k\right)}{N_{k+1}}=0 
	\]
	and let us enumerate the points in $\mathcal{S}_k$ and write them as
	\[
	\mathcal{S}_k=\left\{x_i^k: i=1,2, \ldots, \#\mathcal{S}_k\right\}.
	\]
	
	For each $k\in \N$, let us consider the set of words of length $N_k$ with entries in $\left\{1,2, \ldots,\#\mathcal{S}_k\right\}$. Then, each such word $\underline{i}=\left(i_1, \ldots, i_{N_k}\right)$ represents a point in $\mathcal{S}_k^{N_k}$. Using the specification property, we can choose a point $y:=y\left(i_1, \ldots, i_{N_k}\right)\in X$ such that it $\varepsilon\slash 2^k$-shadows, with gap $m_k$, the pieces of orbits $\{x_{i_j}^k,f(x_{i_j}^k),\dots,f^{n_k-1}(x_{i_j}^k)\}$, $j=1,2,\ldots, N_k$. More precisely, $y$ satisfies
	\[
	d_{n_k}\left(x_{i_j}^k, f^{a_j} (y)\right)<\frac{\varepsilon}{2^k} \;\text{ for all } j \in\left\{1, \ldots, N_k\right\},
	\]
	where $a_j=(j-1)\left(n_k+m_k\right)$. Then, using these points we define
	\[
	\mathcal{C}_k=\left\{y\left(i_1, \ldots, i_{N_k}\right) \in X:\left(i_1, \ldots, i_{N_k}\right) \in\left\{1, \ldots, \#\mathcal{S}_k\right\}^{N_k}\right\} .
	\]
	Moreover, consider $c_k=N_k n_k+\left(N_k-1\right) m_k$. Observe that $c_k$ gives the amount of time for which the orbit of points in $\mathcal{C}_k$ has been prescribed. We now observe that different sequences in $\left\{1, \ldots, \#\mathcal{S}_k\right\}^{N_k}$ give rise to different points in $\mathcal{C}_k$ and that such points are uniformly separated with respect to $d_{c_k}$. 
	
	\begin{lemma}[Lemma 5.1 of \cite{TV}] \label{lem: sep}
		If $(i_1,\dots,i_{N_k})\not=(j_1,\dots,j_{N_k})$, then
		\[
		d_{c_k}\left(y(i_1,\dots,i_{N_k}),y(j_1,\dots,j_{N_k})\right)>\frac{17 \varepsilon}{16}.
		\]
		In particular $\# \mathcal{C}_k=(\#\mathcal{S}_k)^{N_k}$.
	\end{lemma}

	\subsubsection{Construction of $\mathcal{T}_k$} The final intermediate step in the construction of the fractal set $F$ consists in building a third auxiliary family of sets that will be denoted by $\left\{\mathcal{T}_k\right\}_{k \in \mathbb{N}}$. This will be done inductively. Let $\mathcal T_1=\mathcal C_1$. Now, suppose that we have already constructed the set $\mathcal T_k$ and let us describe how to obtain $\mathcal T_{k+1}$. Consider initially $t_1=c_1$ and define $t_{k+1}=t_k+m_{k+1}+c_{k+1}$ for $k\geq 2$. Then, for $x\in \mathcal T_k$ and $y\in \mathcal C_{k+1}$, let $z=z(x,y)$ be some point such that
	\begin{align}\label{eq:4}
		d_{t_k}(x,z)<\frac{\varepsilon}{2^{k+1}} \text{ and } d_{c_{k+1}}(y,f^{t_k+m_{k+1}}(z))<\frac{\varepsilon}{2^{k+1}}.
	\end{align}
	Observe that the existence of such a point is guaranteed by the specification property of $f$. Finally, we define
	\[
	\mathcal T_{k+1}=\{z(x,y):x\in \mathcal T_k, \; y\in \mathcal C_{k+1}\}.
	\]
	
	By proceeding as in the proof of the Lemma \ref{lem: sep} we can see that different pairs $(x,y)$, $x\in \mathcal T_k$, $y\in \mathcal C_{k+1}$, produce different points $z=z(x,y)$. In particular, $\# \mathcal T_{k+1}=\#\mathcal T_k \cdot \#\mathcal C_{k+1}$. Therefore, proceeding inductively,
	\[\# \mathcal{T}_k=\# \mathcal{C}_1 \ldots \# \mathcal{C}_k=   (\#\mathcal S_1)^{N_1} \ldots   (\#\mathcal S_k)^{N_k} .
	\]
	Moreover, by Lemma \ref{lem: sep} and \eqref{eq:4} we have that for every $x\in  \mathcal T_k$ and $y,y'\in \mathcal C_{k+1}$ with $y\not=y'$,
	\begin{align}\label{eq:500}
		d_{t_k}(z(x,y),z(x,y'))<\frac{\varepsilon}{2^{k+2}} \;
		\text{ and } \;   d_{t_{k+1}}(z(x,y),z(x,y'))>\frac{3\varepsilon}{4}.
	\end{align}

	\subsubsection{Construction of $F$} For every $k\in\mathbb N$, let us consider
	\[
	F_k:=\bigcup_{x\in \mathcal T_k}\overline B_{t_k}(x, \varepsilon/2^{k+1}),
	\]
	where $\overline B_{ t_k}(x, \varepsilon/2^{k+1})$ denotes the closure of the open ball $ B_{t_k}(x, \varepsilon/2^{k+1})$. As a simple consequence of \eqref{eq:500} we have the following observation.

	\begin{lemma}[Lemma 5.2 of \cite{TV}] \label{lemma:11}
		For every $k\in\mathbb N$, the following is satisfied:\\
		\noindent(1) for any $x, x'\in \mathcal T_k$, $x \not= x'$, the sets $\overline B_{t_k}(x, \varepsilon/2^{k+1})$ and 
		$\overline B_{t_k}(x', \varepsilon/2^{k+1})$ are disjoint;\\
		\noindent(2) if $z\in \mathcal T_{k+1}$ is such that $z=z(x,y)$ for some $x \in\mathcal T_k$ and $y\in \mathcal{C}_{k+1}$, then
		\[
		\overline B_{t_{k+1}}\left(z, \frac{\varepsilon}{2^{k+2}}\right)\subset \overline B_{t_k}\left(x, \frac{\varepsilon}{2^{k+1}}\right).
		\]
		Hence, $F_{k+1}\subset F_k$.
	\end{lemma}
	
	Then, we define the fractal set $F$ as
	\[F:=\bigcap_{k\in\mathbb N}F_k.\]
	Observe that, since each $F_k$ is a closed and non-empty set and, moreover, $F_{k+1}\subset F_k$, the set $F$ is non-empty and closed itself. Furthermore, using \eqref{eq:201} we may prove the following lemma.
	\begin{lemma}[Lemma 5.3 of \cite{TV}] \label{lemma: F_subset_K_alpha} 
		Under the above conditions,
		\[F\subset K_\alpha.\]
	\end{lemma}

	\subsubsection{Construction of $\eta$}
	Now, in order to construct $\eta$, we present a useful representation of points in $F$. For this purpose we follow closely the material in \cite{Tho,MR2765447}. More precisely, we observe that each point $p \in F$ can be uniquely represented by a sequence $\underline{p}=\left(\underline{p}_1, \underline{p}_2, \underline{p}_3, \ldots\right)$, where $\underline{p}_i=\left(p_1^i, \ldots, p_{N_i}^i\right) \in$ $\left\{1,2, \ldots, M_i\right\}^{N_i}$ (recall that $M_i=\# \mathcal{S}_i$). In fact, each point in $\mathcal{T}_k$ can be uniquely represented by a finite word $\left(\underline{p}_1, \ldots, \underline{p}_k\right)$. This can be seen using the following notation. Let $y(\underline{p}_i) \in \mathcal{C}_i$ and consider $z_1(\underline{p})=y (\underline{p}_1 )$. Then, proceeding inductively, let $z_{i+1}(\underline{p})=z (z_i(\underline{p}), y (\underline{p}_{i+1} ) ) \in \mathcal{T}_{i+1}$. Note that we can also write $z_i(\underline{p})$ as $z (\underline{p}_1, \ldots, \underline{p}_i )$. Then define 
	\[
	p:=\bigcap_{i \in \mathbb{N}} \overline  B_{t_i}\left(z_i(\underline{p}), \frac{\varepsilon }{2^{i-1}}\right) .
	\]
	Thus, from this construction we can see that points in $F$ can be uniquely represented in this way. 
	
	The desired measure $\eta$ will be obtained as the limit of a sequence of atomic measures that we build now. Given $z=z (\underline{p}_1, \ldots \underline{p}_k ) \in \mathcal{T}_k$, we define
	\[
	\mathcal{L}(z):=\mathcal{L} (\underline{p}_1 ) \cdot \ldots \cdot \mathcal{L} (\underline{p}_k ),
	\]
	where, if $\underline{p}_i=\left(p_1^i, \ldots, p_{N_i}^i\right) \in\left\{1, \ldots,  \#\mathcal{S}_i\right\}^{N_i}$, then
	\[
	\mathcal{L} (\underline{p}_i ):=\prod_{l=1}^{N_i} \exp \left( S_{n_i} \psi_{5\varepsilon/4} \left(x_{p_l^i}^i\right) \right).
	\]
	Using these numbers we define
	\[
	\nu_k:=\sum_{z \in \mathcal{T}_k} \delta_z \mathcal{L}(z) 
	\]
	and $\mu_k:=$ $\left(1 / \kappa_k\right) \nu_k$, where $\kappa_k$ is the normalising constant given by
	\[
	\kappa_k:=\sum_{z \in \mathcal{T}_k} \mathcal{L}_k(z) .
	\]
	Observe that $\kappa_k=M_1^{N_1} \cdot \ldots \cdot M_k^{N_k}$ (see \cite[Lemma 3.9]{MR2765447}). The next lemma shows us that any limiting measure of the sequence $(\mu_k)_{k\in \N}$ lives on $F$. Such a measure is the desired measure that we were looking for.
	
	\begin{lemma}[Lemma 3.10 of \cite{MR2765447}] \label{lem: eta in F}
		Suppose that $\eta$ is a limit measure of the sequence of probability measures $(\mu_k)_{k\in \N}$. Then $\eta(F)=1$.
	\end{lemma}
	
	In fact, by proceeding as in the proof of \cite[Lemma 5.4]{TV}, we can prove that the sequence $(\mu_k)_{k\in \N}$ actually converges in the weak$^\ast$-topology, but this fact is not needed in the sequel. An important feature of the measure $\eta$ given in Lemma \ref{lem: eta in F} and that can be obtained by exploring its definition and \eqref{eq:200} is that the $\eta$-measure of some appropriate dynamical balls decay exponentially fast. More precisely,
	
	\begin{lemma}[Lemma 3.17 of \cite{MR2765447} and Lemma 3.20 of \cite{Tho}] \label{lem: decay}
		Let $C:=h_{\mu}(f,\xi)+\int \psi_{5\varepsilon/4}  d\mu$ and $D:=\operatorname{Var}(\psi_{5\varepsilon/4} , 2\varepsilon )$. Then, for any $n$ sufficiently large and $q\in X$ such that $B_n\left(q,\frac{\varepsilon}{2}\right)\cap F\neq \emptyset$, we have that
		\[
		\limsup_{l\to\infty}\mu_{l}\left(B_n\left(q,\frac{\varepsilon}{2}\right)\right)
		\leq \exp \left(-n (1-\gamma)^2\left(C-5\gamma+2D \right)+ S_n  \psi_{5\varepsilon/4}(q) \right).
		\]
	\end{lemma}

	In order to conclude our proof we need the following fact which is sometimes refereed to as the generalized \emph{Pressure Distribution Principle}  (see \cite[Proposition 3.2]{MR2765447}). Observe that for this result, the measures involved do not need to be invariant by $f$, as it is the case of the measures $\mu_k$ and $\eta$ obtained in the previous constructions.

	\begin{lemma}\label{lem:pressure distribution}
		Let $f: X \to X$ be a continuous transformation and $Z \subseteq X$ be an arbitrary Borel set. Suppose that there exists $\varepsilon >0$ and $s \geqslant 0$ such that one can find a sequence of Borel probability measures $\mu_k$, a constant $K>0$ and an integer $N$ satisfying
		\[
		\limsup _{k \rightarrow \infty} \mu_k\left(B_n(x,\varepsilon )\right) \le  K \exp \left(-n s+\sum_{i=0}^{n-1} \psi\left(f^i (x)\right)\right)
		\]
		for every ball $B_n(x,\varepsilon)$ such that $B_n(x,\varepsilon)\cap Z\neq \emptyset$ and $n\geq N$. Furthermore, assume that at
		least one limit measure $\eta$ of the sequence $\mu_k$ satisfies $\eta(Z)>0$. Then $m(Z, f, \psi , \varepsilon)>s$.
	\end{lemma} 
	
	Combining Lemmas \ref{lem: eta in F}, \ref{lem: decay} and \ref{lem:pressure distribution} we get that
	\[m(F, f, \psi_{\varepsilon/2} ,  \varepsilon/2)>(1-\gamma)^2(C-5\gamma+2D).\]
	On the other hand, by Lemma \ref{lemma: F_subset_K_alpha} it follows that
	\[m( K_\alpha, f, \psi_{\varepsilon/2} ,  \varepsilon/2)\geq m(F, f, \psi_{\varepsilon/2} ,  \varepsilon/2).\]
	Consequently, using \eqref{eq: choice varepsilon for P}, \eqref{eq: coiche vareps gamma} and \eqref{eq:escolha de xi}, we obtain that
\begin{align*}
			\mathrm{\overline{mdim}_M^B}\,(K_\alpha,f,d,\psi) +\gamma &\ge  \frac{m( K_\alpha, f, \psi_{\varepsilon/2} ,  \varepsilon/2)}{|\log \varepsilon/2|}\\
			&\ge \frac{(1-\gamma)^2(C-5\gamma+2D)}{|\log \varepsilon/2|}\\
			&=(1-\gamma)^2\frac{|\log 5\varepsilon|}{|\log \varepsilon/2|}\frac{h_{\mu}(f,\xi)+\int \psi_{5\varepsilon}d\mu-5\gamma}{|\log 5\varepsilon|}\\
			&\phantom{=}  -(1-\gamma)^2\left|\frac{\log 4 \int\psi d\mu-2  (\log5\varepsilon/4)\operatorname{Var}(\psi, 2\varepsilon )}{\log \varepsilon/2}\right|\\
			& \ge(1-\gamma)^3 \mathrm {H_\varphi\overline{mdim}_M}\,(f,\alpha,d,\psi)-2\gamma.
	\end{align*}
	Thus, since $\gamma >0$ is arbitrary, we conclude that
	\[\mathrm{\overline{mdim}_M^B}\,(K_\alpha,f,d,\psi)\geq \mathrm{H_\varphi\overline{mdim}_M}\,(f,\alpha,d,\psi).\]

	In order to establish the second lower bound for $\mathrm{\overline{mdim}_M}\,(K_\alpha,f,d,\psi)$, we fix $\gamma>0$ and let $\mu\in\mathcal M_f(X,\varphi,\alpha)$ be so that 
	\begin{align*}
		H(\mu)+ \int\psi d\mu \geq\sup\left\{H(\nu)+\int_X \psi\ d\nu : \nu \in M_f(X,\varphi,\alpha)\right\}-\gamma.
	\end{align*}
	Moreover, let $(\mu_\varepsilon)_\varepsilon\subset \mathcal M(\mu)$ be such that 
	\begin{align*}
		\limsup_{\varepsilon\to0}\frac{\inf_{|\xi|<\varepsilon}h_{\mu_\varepsilon}(f,\xi)}{|\log\varepsilon|}+ \int\psi\ d\mu \geq H(\mu)+ \int\psi\ d\mu-\gamma
	\end{align*}
	and take $\varepsilon=\varepsilon(\gamma)>0$ for which 
	\begin{align*}
		\frac{\inf_{|\xi|<5\varepsilon}\left( h_{\mu_{5\varepsilon}}(f,\xi) +\int\psi_{5\varepsilon}\ d\mu\right)-5\gamma}{|\log5\varepsilon|} \geq  H(\mu)+ \int\psi\ d\mu-2\gamma,
	\end{align*}
	\begin{equation*}\label{eq:choice varepsilon for h}
		\frac{m( K_\alpha, f, \psi_{\varepsilon/2} ,  \varepsilon/2)}{|\log \varepsilon/2|}\leq \mathrm{\overline{mdim}_M}\,(K_\alpha,f,d,\psi) +\gamma,
	\end{equation*}
	and
	\begin{equation*}
		\frac{|\log 5\varepsilon|}{|\log \varepsilon/2|}\geq (1-\gamma) \; \text{ and }\; \left| \frac{\log 4 \int\psi d\mu_\varepsilon-2  (\log5\varepsilon/4)\operatorname{Var}(\psi, 2\varepsilon )}{\log \varepsilon/2}\right|<\gamma.
	\end{equation*}
	Then, proceeding as we did above we conclude that for some partition $\xi$ of $X$,
\begin{align*}
			\mathrm{\overline{mdim}_M^B}\,(K_\alpha,f,d,\psi) +\gamma &\ge  \frac{m( K_\alpha, f, \psi_{\varepsilon/2} ,  \varepsilon/2)}{|\log \varepsilon/2|}\\
			&\ge \frac{(1-\gamma)^2(h_{\mu_{5\varepsilon}}(f,\xi)+\int \psi_{5\varepsilon/4} d\mu-5\gamma+2  \operatorname{Var}(\psi_{5\varepsilon/4}, 2\varepsilon ))}{|\log \varepsilon/2|}\\
			&=(1-\gamma)^2\frac{|\log 5\varepsilon|}{|\log \varepsilon/2|}\frac{h_{\mu_{5\varepsilon}}(f,\xi)+\int \psi_{5\varepsilon}d\mu-5\gamma}{|\log 5\varepsilon|}\\
			&\quad\quad-(1-\gamma)^2\left|\frac{\log 4 \int\psi d\mu-2  (\log5\varepsilon/4)\operatorname{Var}(\psi, 2\varepsilon )}{\log \varepsilon/2}\right|\\
			& \ge(1-\gamma)^3\sup\left\{H(\nu)+\int_X \psi\ d\nu:  \nu \in M_f(X,\varphi,\alpha)\right\}-3\gamma.
	\end{align*}
	Consequently, since $\gamma>0$ is arbitrary, it follows that
	\[ \mathrm{\overline{mdim}_M^B}\,(K_\alpha,f,d,\psi)\geq \sup\left\{H(\nu)+\int_X \psi\ d\nu:  \nu \in M_f(X,\varphi,\alpha)\right\},\]
	giving us the second lower bound for $\mathrm{\overline{mdim}_M^B}\,(K_\alpha,f,d,\psi)$.

	\subsection{Conclusion of the proof}\label{sec: conclusion of proof}
	In order to finish the proof of Theorem \ref{thm:main1}, all that is left to do is to show that 
	\[
	\begin{split}
		\sup\left\{H(\mu)+\int_X \psi\ d\mu:  \mu \in \mathcal{M}_f(X,\varphi,\alpha)\right\}=\sup\left\{H_\delta^K(\mu)+\int_X \psi\ d\mu:\mu\in \mathcal M_f(X,\varphi , \alpha)\right\}.
	\end{split}
	\]
	By the proof of \cite[Theorem 9]{Shi}, given $\delta\in(0,1)$ we have that for any  $\nu\in\mathcal{M}_f^e(X)$ and any finite measurable partition $\xi$ with diameter smaller than $\frac{\varepsilon}{4}$, 
	\[
	h_{\nu}^K(\varepsilon,\delta)\leq h_{\nu}(f,\xi).
	\]
	Consequently, using property \eqref{convexity-property} and an analogous property for the metric entropy of a measurable partition together with the definitions of $H_\delta^K(\cdot)$ and $H(\cdot)$ we obtain 
	\[
	\begin{split}
		\sup\left\{H_\delta^K(\mu)+\int_X \psi\ d\mu:\mu\in \mathcal M_f(X,\varphi , \alpha)\right\}\leq\sup\left\{H(\mu)+\int_X \psi\ d\mu:  \mu \in \mathcal{M}_f(X,\varphi,\alpha)\right\}.
	\end{split}
	\]
	
	For the converse inequality, again by the proof of \cite[Theorem 9]{Shi} we get 
	\begin{align*}
		\inf_{|\xi|<\varepsilon} h_\nu(f,\xi)\leq h_\nu^K\left(\frac{\varepsilon}{4},\delta\right),
	\end{align*}
	for any $\delta\in(0,1)$ and any $\mu\in\mathcal M_f(X)$, which implies the desired conclusion. This completes the proof of Theorem \ref{thm:main1}.

	\section{Application to suspension flows}
	Inspired by the results in \cite{Tho}, we apply our main theorem to the context of suspension flows. We start by recalling the definition of metric mean dimension for flows.
	
	\subsection{Metric mean dimension for flows}
	Given a compact metric space $(X,d)$, let $g_t\colon X\to X$, $t\in \R$ be a continuous flow which we sometimes simply denote by $g$. In particular, it satisfies $g_{t+s}=g_t\circ g_s$ for all $t,s \in \mathbb{R}$. Similarly to what we did in the case of discrete time dynamics, we define the \emph{Bowen metric for the flow $g_t$} as
	\begin{align*}
		d_t(x,y)=\max_{s\in [0,t]} d(g_s (x), g_s (y))
	\end{align*}
	and consider the $(t,\epsilon)$-ball around $x$ given by  
	\begin{align*}
		B_t(x,\epsilon)&=\{y\in X: d_t(x,y)<\epsilon\}.
	\end{align*}
	Clearly, since $g_t$ is continuous, $B_t(x,\epsilon)$ is an open set.

	Given a set $Z\subset X$, let us consider
	\[
	M(Z,g, s, T,\varepsilon)=\inf_{\Gamma}\left\{\sum_{i\in I}\exp{\left(-st_i\right)}\right\},
	\]
	where the infimum is taken over all covers $\Gamma=\{B_{t_i}(x_i,\varepsilon)\}_{i\in I}$ of $Z$ with $t_i\geq T$. Moreover, we consider
	\[
	M(Z,g, s,\varepsilon)=\lim_{T\to\infty}M(Z,g, s, T,\varepsilon)
	\]
	and define
	\begin{align}
		M\Big(Z,g,\varepsilon\Big)&=\inf\{s:M(Z,g, s,\varepsilon)=0\}
		\nonumber\\
		&=\sup\{s:M(Z,g, s,\varepsilon)=+\infty\}.
	\end{align}
	See \cite{Pesin} for details. Then, the \textit{upper Bowen metric mean dimension of $g$ on $Z$} is defined as
		\begin{equation*}
			\mathrm{\overline{mdim}_M^B}\,\Big(Z,g,d\Big)=\limsup_{\varepsilon\to 0}\frac{M\Big(Z,g ,\varepsilon\Big)}{|\log \varepsilon|}.
	\end{equation*}

	\subsection{Suspension flows}
	Let $f\colon X\to X$ be a homeomorphism of a  compact metric space $(X,d)$. Given a continuous roof function $\rho \colon X\to(0,\infty)$, we consider the \emph{suspension space} or \emph{quotient space} given by 
	\[X_\rho=\{(x,s)\in X\times \mathbb R: 0\leq s\leq \rho(x)\}/\sim ,\]
	where $\sim$ is the equivalence relation given by $(x,\rho(x))\sim (f(x),0)$ for all $x\in X$. The  \emph{suspension of $(X,d)$ with roof function $\rho$} is the semiflow $\Psi=\{g_t\}_{t\in\mathbb R}$ on $X_\rho$ defined locally by
	$g_t(x,s)=(x,s+t)$. 
	Given a continuous function $\Phi\colon X_\rho\to \mathbb R$, we associate to it the continuous function $\varphi\colon X\to\mathbb R$ defined by $\varphi(x)=\int_0^{\rho(x)}\Phi(x,t)\ dt$. Then, by \cite[Lemma 5.3]{MR2765447}
	we have that 
	\[
	\liminf_{T\to\infty}\frac{1}{T}\int_0^T\Phi(g_t(x),t)\ dt=\liminf_{n\to\infty}\frac{S_n\varphi(x)}{S_n\rho(x)}\]
	and
	\[
	\limsup_{T\to\infty}\frac{1}{T}\int_0^T\Phi(g_t(x),t)\ dt=\limsup_{n\to\infty}\frac{S_n\varphi(x)}{S_n\rho(x)}.\]
	Moreover, to any $\mu\in \mathcal M_f(X)$ we associate the measure $\mu_\rho$ on $X_\rho$ given by 
	\[\int_{X_\rho}\Phi \ d\mu_\rho=\frac{\int_{X}\varphi \ d\mu}{\int_X \rho \ d\mu}, \;\text{ for every }  \Phi\in C(X_\rho,\R),\]
	where $\varphi$ is defined as above. It is well known that all the $\Psi $-invariant measures can be obtained via this construction, that is, 
	\[\mathcal M_\Psi(X_\rho)=\left\{\mu_\rho:\mu\in \mathcal{M}_f(X)\right\}.\]
	Furthermore, by Abramov's Theorem \cite{Abra} we have that $h_{\mu_\rho}(\Psi)=\displaystyle\frac{h_\mu(f)}{\int_X\rho(x)\ d\mu}$ and hence 
	\[h_\mathrm{{top}}(\Psi)=\sup\left\{h_{\mu_\rho}(\Psi):\mu\in\mathcal M_f(X)\right \}=\sup\left\{\displaystyle\frac{h_\mu(f)}{\int_X\rho(x)\ d\mu}:\mu\in\mathcal M_f(X) \right\}\]
	where $h_\mathrm{{top}}(\Psi)$ denotes the topological entropy of the flow $\Psi$. Note that we use $h_\mathrm{{top}}(Z,\Psi)$ to denote the topological entropy of a subset $Z$ for the flow  $\Phi$, as introduced by \cite{MR2765447}. In particular, we have $h_\mathrm{{top}}(\Psi)=h_\mathrm{{top}}(X,\Psi)$.
	
	Finally, given $\alpha \in \R$ we consider the level set
	\begin{align*}
		X_\rho(\Phi,\alpha)&=\left\{(x,s)\in X_\rho: \lim_{T\to\infty}\frac{1}{T}\int_0^T\Phi(g_t(x),t)\ dt=\alpha\right\}\\
		&=\left\{(x,s)\in X_\rho: \lim_{n\to\infty}\frac{S_n\varphi(x)}{S_n\rho(x)}=\alpha\right\}.
	\end{align*}
	In \cite[Theorem 4.2]{Tho}, it was proved that for $\Phi\in C(X_\rho,\R)$, 
	\[h_{\mathrm{top}}(X_\rho(\Phi,\alpha),\Psi)=\sup\left\{h_{\mu}(\Psi):\mu\in \mathcal M_\Psi(X_\rho)\text{ and }\int\Phi\ d\mu=\alpha\right\}.\]

	\subsection{Metric mean dimension of suspension flows}
	We start with a version of \cite[Theorem 4.1]{Tho} in our setting. 
	
	\begin{theorem}\label{thm:10101010}
		Let $(X,d)$ be a compact metric space, $f\colon X\to X$ be a continuous map with the specification property and $\varphi,\psi, \rho \in C(X,\R) $ with $\rho>0$. Given $\alpha\in \R$, consider $X(\varphi,\rho,\alpha)=\left\{x\in X:\lim_{n\to\infty}\frac{S_n\varphi(x)}{S_n\rho(x)}=\alpha\right\}$. If $\alpha$ is so that $X(\varphi,\rho,\alpha)\not=\emptyset$, then
		\[\mathrm{\overline{mdim}_M^B}\,(X(\varphi,\rho,\alpha),f,d,\psi)=\sup\left\{H(\mu)+\int_X\psi d\mu: \; \mu\in\mathcal M_f(X)\text{ and }\frac{\int_{X}\varphi \ d\mu}{\int_X \rho \ d\mu}=\alpha\right\}.\]
	\end{theorem}
	\begin{proof}
		Observe initially that the proof of Lemma \ref{lem: approx} may be slightly changed in order to obtain a measure $\nu_k$ so that instead of getting the item (c) in the aforementioned lemma we obtain 
		\[\left|\displaystyle\frac{\int_X \varphi\; d\nu_k}{\int_X \rho\; d\nu_k}-\frac{\int_X \varphi\; d\mu}{\int_X \rho\; d\mu}\right|<\delta_k.\]
		Moreover, by Hopf’s ratio ergodic theorem it is possible to replace the family of sets given in \eqref{def:sets} by the following family of sets 
		\[Y_{i}(k)=\left\{x\in X: \left|\frac{S_n\varphi(x)}{S_n\rho(x)}-\frac{\int_X\varphi\; d\nu^k_i}{\int_X\rho\; d\nu^k_i}\right|<\delta_k\;\; \forall \; n\geq \ell_k\right\},\]
		which for $\ell_k$ large enough has $\nu_i^k$-measure bigger than $1-\gamma$ for every $k\in\mathbb N$ and $i\in \{1,\dots,j(k)\}$. Then, after these two small changes, we can follow the steps of the proof of Theorem \ref{thm:main1} to obtain the desired result.
	\end{proof}
	
	In order to talk about the metric mean dimension of a suspension flow,  we obviously need a metric in the suspension space. Thus we will consider $X_\rho$ endowed with the \emph{Bowen-Walters metric} $\tilde d$ which is induced by $d$ on $X_\rho$ (see \cite{BW,Gu} and \cite[Section 2.2]{Barreira2} for more details). For the sake of completeness as well as the convenience of the later proof,  we recall the construction. First we assume $\rho\equiv 1$  and define the metric $\tilde{d}_{1}$ on the space $X_1$. For $x, y \in X$ and $0 \leq t \leq 1$, define the length of the horizontal segment $((x, t),(y, t))$ by
	\[
	d_h((x, t),(y, t))=(1-t) d(x, y)+t d(f x, f y) .
	\]
	Then, for $(x, t),(y, s) \in  X_1$ which are on the same orbit, define the length of the vertical segment $((x, t),(y, s))$ by
	\[
	d_v((x, t),(y, s))=\inf \left\{|r|: g_r(x, t)=(y, s)\right\} .
	\]
	Finally, for any $(x, t),(y, s) \in X_1$, define the distance $\tilde{d}_{1}((x, t),(y, s))$ to be the infimum of the lengths of all paths between $(x, t)$ and $(y, s)$ consisting of a finite concatenation of horizontal and vertical segments. 
	
	Now we consider the case when $\rho: X \rightarrow(0, \infty)$ is a general continuous function. For this purpose, observe that there is a natural homeomorphism $i_\rho:  X_1 \to X_\rho$ given by $(x, t) \mapsto(x, t \rho(x))$. Using this homeomorphism we define the Bowen-Walters metric $\tilde{d}$ on $X_\rho$ as 
	\begin{equation}\label{eq:BW-metric}
		\tilde{d}((x, t),(y, s))=\tilde{d}_{1}\left(i_\rho(x, t),i_\rho(y, s)\right)=\tilde{d}_{1}\left((x, t\rho(x),(y, s\rho(y)\right) .
	\end{equation}

	Let $(x,s)\in X_\rho$ with $0\leq s<\rho(x)$. Following \cite{MR2765447}, we consider the  \emph{horizontal segment of $(x,s)$} given by
	$\{(y,t): y\in X, \ 0\leq t<\rho(x),\ t=\rho(y)s\rho(x)^{-1}\}$ and the \emph{horizontal ball of radius $\varepsilon$ at
		$(x,s)$} defined by 
	\[B^H\left((x,s),\varepsilon\right):=\left\{\left(y,\frac{s}{\rho(x)}\rho(y)\right):\left(1-\frac{s}{\rho(x)}\right)d(x,y)+\frac{s}{\rho(x)}d(f(x),f(y))<\varepsilon\right\}.\]
	We also define for $(x,s)\in X_\rho$ and $T,\varepsilon>0$,
	\[
	B((x,s),\varepsilon)=\bigcup_{t:|s-t|<\varepsilon}B^H((x,t),\varepsilon)\]
	and
	\[
	B_T((x,s),\varepsilon)=\bigcap_{t=0}^Tg_{-t}(B(g_t(x,s),\varepsilon)).
	\]
	We emphasize that $ B((x,s),\varepsilon)$ is not necessarily a ball in the Bowen-Walters metric. But, on the other hand, for every $\varepsilon>0$ there exist constants $C_1,C_2>0$ such that the following holds (see Section 5.4 in \cite{MR2765447}):
	\begin{itemize}
		\item the ball in the Bowen-Walters metric with center $(x,s)$ and radius $C_1\varepsilon$ is a subset of $B((x,s),\varepsilon)$;
		\item every subset of $X_\rho$ with diameter $\varepsilon$ in the Bowen-Walters metric is contained is some set of the form $B((x,s),C_2\varepsilon)$, for $\varepsilon$ sufficiently small;
		\item $B((x,s),\varepsilon)$ is open with respect to the topology induced by Bowen-Walters metric; 
		\item $\displaystyle\lim_{\varepsilon \to 0} \sup_{(x,s)\in X_\rho}\text{diam}(B((x,s),\varepsilon))=0$, where $\text{diam}(A)$ is the diameter of the subset $A$ of $X_\rho$ with respect to Bowen-Walters metric.
	\end{itemize}
	These properties will allow us to compute the metric mean dimension  with potential of $X_\rho $ endowed with the Bowen-Walters metric by using open covers given by open sets of the form $B((x,s),\varepsilon)$. 
	
	Fix $K=4\ \|\rho\|\slash \inf \rho$. Combing Lemmas 5.6 and 5.7 in \cite{MR2765447} we can conclude that for $\varepsilon>0$ satisfying $K\varepsilon<\inf \rho$, if $|s|<\varepsilon$ and  $S_n\rho(x)\leq T<S_{n+1}\rho(x)$, then \begin{align}
		\label{eq:irre-thomp}
		B_T((x,s),\varepsilon
		)\subset B_n(x,K\varepsilon)\times (-K\varepsilon,K\varepsilon).
	\end{align}

	\begin{theorem}\label{thm2020202}
		Let $(X,d)$ be a compact metric space and $f\colon X\to X$ be a homeomorphism. Let $\rho\colon  X\to(0,\infty)$ be a continuous function and $(X_\rho,\Psi)$ be the corresponding suspension flow over $X$. Given $Z\subset X$, define $Z_\rho=\{(z,s):z\in Z\text{ and }0\leq s <\rho(z)\}$. Let $\beta\in \mathbb R$ be the unique solution of the equation   $\mathrm{\overline{mdim}_M^B}\,(Z, f,d,-t\rho)=0$. Then $\mathrm{\overline{mdim}_M^B}\,(Z_\rho,\Psi, \tilde d)\geq \beta$.
	\end{theorem}
	\begin{proof}
		We start observing that, by \cite[Proposition 3.6]{Chen}, the map $$t\mapsto \mathrm{\overline{mdim}^B_M}\,(Z,f, d,-t\rho)$$ is decreasing and continuous and, moreover, there exists a unique value of $t$ for which $\mathrm{\overline{mdim}_M^B}\,(Z,f, d,-t\rho)=0$. Thus, in order to get our result it is enough to prove that if $\beta>0$ is so that $\mathrm{\overline{mdim}_M^B}\,(Z,f, d,-\beta\rho)>0$, then $\mathrm{\overline{mdim}_M^B}\,(Z_\rho,\Psi,\tilde d)\geq \beta$. This is what we are going to do in the sequel.
		
		Let $\varepsilon>0$ be small enough such that $\eqref{eq:irre-thomp}$ holds and $\displaystyle\frac{m\,\Big(Z,f,-\beta \rho,\varepsilon\Big)}{|\log \varepsilon|}>0$. Observe that there are arbitrary small values of $\varepsilon>0$ satisfying these two conditions. Then, consider an open cover of $Z_\rho$ of the form $\Gamma=\{B_{t_i}((x_i,s_i),\varepsilon)\}_i$ with $t_i\geq T$ and $s_i\leq \varepsilon$ for all $i$. Since $Z\times \{0\}\subset Z_\rho$, we can take $\Gamma'$ a subcover of $\Gamma$ which covers $Z\times \{0\}$. Denote by $m_i\in \mathbb N$ the unique number so that $S_{m_i}\rho(x_i)\leq t_i<  S_{m_i+1}\rho(x_i)$ and consider
		$m(\Gamma')=\inf m_i$. By the definition of the $m_i$'s and the way the $t_i$'s were taken, we have that $m(\Gamma')\geq \|\rho\|^{-1}(T-\|\rho\|)$ and, consequently, $\displaystyle\lim_{T\to +\infty}m(\Gamma')=\infty$. 
		
		Let $\Gamma''=\{B_{m_i}(x_i,K\varepsilon):B_{t_i}((x_i,s_i),\varepsilon)\in\Gamma'\}$. By \eqref{eq:irre-thomp} it follows that the family given by the sets of the form $B_{m_i}(x_i,K\varepsilon)\times (-K\varepsilon ,K\varepsilon)$ covers $Z\times \{0\}$. In particular, $\Gamma''$ covers $Z$. Now we notice that 
		\begin{align}\label{ineq:ai-vai}
			&\sum_{B_{t_i}((x_i,s_i),\varepsilon)\in\Gamma'}\exp{\left(-\beta |\log\varepsilon| t_i\right)}\\ \nonumber
			&\geq
			\sum_{B_{t_i}((x_i,s_i),\varepsilon)\in\Gamma'}\exp{\left(-\beta S_{m_i+1}\rho_\varepsilon(x_i)\right)}\\ \nonumber
			&\geq 
			\sum_{B_{t_i}((x_i,s_i),\varepsilon)\in\Gamma'}\exp{\left(-\beta (S_{m_i}\rho(x_i)+\|\rho\|)_\varepsilon\right)}\\ \nonumber
			&\geq 
			\sum_{B_{m_i}((x_i,\varepsilon)\in\Gamma''}\exp{\left(-\beta \left(\inf_{y\in B_{m_i}(x_i,K\varepsilon)}S_{m_i}\rho(y)+\text{Var}(\rho,\varepsilon)+\|\rho\|\right)_\varepsilon\right)}\\\nonumber
			&= 
			\exp(-\beta(\text{Var}(\rho,\varepsilon)+\|\rho\|)_\varepsilon)
			\sum_{B_{m_i}(x_i,\varepsilon)\in\Gamma''}\exp{\left(\sup_{y\in B_{m_i}(x_i,K\varepsilon)}S_{m_i}(-\beta\rho)_\varepsilon(y)\right)}\\\nonumber
			&\geq
			\exp(-\beta(\text{Var}(\rho,\varepsilon)+\|\rho\|)_\varepsilon) m(Z,-\beta\rho,0,m(\Gamma'),\varepsilon).
		\end{align}   
		Thus, by taking $T$ sufficiently large (recall that this implies that $m(\Gamma')$ also gets large) and using the hypothesis $\displaystyle\frac{m\,\Big(Z,f,-\beta \rho,\varepsilon\Big)}{|\log \varepsilon|}>0$, we get that
		\[\exp(-\beta(\text{Var}(\rho,\varepsilon)+\|\rho\|)_\varepsilon) m(Z,-\beta\rho,0,m(\Gamma'),\varepsilon)\geq 1.\]
		Therefore, using \eqref{ineq:ai-vai} we conclude that
		\[\sum_{B_{t_i}((x_i,s_i),\varepsilon)\in\Gamma}\exp{\left(-\beta |\log \varepsilon |t_i\right)}\geq\sum_{B_{t_i}((x_i,s_i),\varepsilon)\in\Gamma'}\exp{\left(-\beta |\log \varepsilon | t_i\right)}\geq 1.\]
		Finally, since $\Gamma$ was taken arbitrary, we have that $M(Z_\rho,\Psi, \beta|\log \varepsilon|, T,\varepsilon)\geq 1$, which implies $M(Z_\rho,\Psi,\varepsilon)\geq \beta|\log \varepsilon |$. Since  $\varepsilon>0$ may be taken arbitrary small,
		$\mathrm{\overline{mdim}_M^B}\,(Z_\rho,\Psi, \tilde d)\geq \beta$ as claimed.  \end{proof}
	
	Our next result deals with the metric mean dimension of the level sets $X_\rho(\Phi,\alpha)$. It represents an extension of \cite[Theorem 4.2]{Tho} to the context of infinite entropy. But before we state the main result of this section, we need to define an analogous version of the map $H(\mu)$ for continuous time dynamics. We emphasize that there are several candidates to play the role of this map but we will consider only one of them. Namely, we will define the map $H(\mu)$ associated to a flow $\Psi=\{g_t\}$ as the map $H(\mu)$ associated to the time-one map $g_1$; we abuse notation and denote both maps in the same way.

	For the next result we assume that $\rho:X\to[0,\infty)$
	is so that the following condition holds: 
	there exists  $K_1>0$  such that for any $x,y\in X$,
	\begin{equation} \label{eq:important-condition}
		|S_n\rho(x)-S_n\rho(y)|\leq K_1 d_n(x,y).
	\end{equation}
	\begin{remark}
		Observe that whenever $\rho$ is Lipschitz continuous, condition \eqref{eq:important-condition} is satisfied.
	\end{remark}
	Now we establish a relation between the measure-theoretic metric mean dimension  of a system and its suspension flow. 
	\begin{theorem}\label{thm:byproduct}
		Given $\mu\in\mathcal M_f(X)$ we have that
		\begin{equation}\label{eq: metric entropy suspension}
			H(\mu_\rho)=\frac{H(\mu)}{\int\rho\ d\mu}.
		\end{equation}
	\end{theorem}
	\begin{proof}
		We start proving that this result holds for $\rho$ satisfying $\rho\ge 1$. So, assume $\rho \geq 1$ and consider 
		\[A=\{(x,t)\in X_\rho:0\le t\le1\}\subset X_\rho.\]
		It is easy to see that $\mu_\rho(A)>0$. Let $\Psi_A$ be the first return map to the set $A$. Then 
		\[\Psi_A(x,t)=(f(x),t-\rho(x)\pmod 1)\text{ for any }(x,t)\in A.\]
		First of all we notice that    $(\mu_\rho)_A=\mu\times\text{Leb}$,   where by $\text{Leb}$ we mean the Lebesgue measure on $[0,1]$. In fact, given $B\in\mathcal B(X)$ and $I\in\mathcal B([0,1])$, we have 
		\[(\mu_\rho)_A(B\times I)=\frac{\int_X\int_0^{\rho(x)}1_B(x)1_I(t)\ dt\ d\mu}{\mu_\rho(A)\int\rho \ d\mu}=\frac{\mu(B)\times \text{Leb}(I)}{\mu_\rho(A)\int\rho \ d\mu}.\]
		Thus, as 
		\[\mu_\rho(A)=\frac{\int_X\int_{0}^{\rho(x)}1_X1_{[0,1]}dtd\mu}{\int \rho d\mu}=\frac{1}{\int\rho \ d\mu},\]
		we obtain that $(\mu_\rho)_A(B\times I)=(\mu\times \text{Leb})(B\times I)$, for any $B\in\mathcal B(X)$ and $I\in\mathcal B([0,1])$.
		
		Observing that $\Psi_A$ is a skew product map generated by $f$ and the family of continuous maps $S={g_x:[0,1]\to[0,1]}$, where $g_x:t\in [0,1]\mapsto t-\rho(x) \pmod 1\in [0,1]$, for any $x\in X$, by \cite[Proposition 1.3 of Chapter 6]{Petersen} 
		we have that for any measurable product partition $\xi=\alpha\times \beta$ of $X\times [0,1]$ and any $\nu\in \mathcal M_f(X)$
		\[h_{\nu\times \text{Leb}}(\Psi_A,\xi)\geq h_{\nu}(f,\alpha)+h_{\nu\times\text{Leb}}(\Psi_A|S,\beta),\]
		where $h_{\nu\times\text{Leb}}(\Psi_A|S,\beta)$ denotes the fiber entropy of the family $S$ with respect to $\nu\times\text{Leb}$. Once $h_{\nu\times\text{Leb}}(\Psi_A|S,\beta)=0$, and  for any family $(\mu_\varepsilon)_\varepsilon\in\mathcal M(\mu)$ we have $(\mu\times\text{Leb})_\varepsilon\in\mathcal M((\mu_\rho)_A)$, 
		\begin{align*}
			\limsup_{\varepsilon\to0} \frac{\inf_{|\xi|<\varepsilon}h_{(\mu_{\rho})_A}(\Psi_A,\xi)}{-\log\varepsilon}\geq \limsup_{\varepsilon\to0} \frac{ \inf_{|\alpha|<\varepsilon} h_{\mu}(f,\alpha)}{-\log\varepsilon}.
		\end{align*}
		So, $H((\mu_\rho)_A)\geq H(\mu)$. 
		
		For the converse inequality we will make use of the notion of Katok's metric entropy. Let $\delta>0$ and consider  $(\bar\mu_\varepsilon)_\varepsilon\in \mathcal M((\mu_\rho)_A)$. For each $\varepsilon>0$, take  $F\subset X$ so that 
		\[
		\nu_\varepsilon(F)=(\pi_1)_\ast(\bar\mu_\varepsilon)(F)=\bar\mu_\varepsilon(\pi_1^{-1}(F))>1-\delta
		\]
		where $\pi_1$ is the projection map to $X$.
		
		Let $M:=\max_{x\in X}|\rho(x)|$. Note that for any $(x,t),(y,s)\in A$, $(x,t\rho(x))$ and $(y,s\rho(y))$ can be connected by the vertical segment $((x,t\rho(x)),(x,s\rho(y))$ and the horizontal segment $((x,s\rho(y)),(y,s\rho(y))$ (see the construction of Bowen-Walters metric above). Thus, by  \eqref{eq:BW-metric},
		\begin{align*}
			\tilde{d}((x,t),(y,s))\le &|t\rho(x)-s\rho(y)|+\max\{d(x,y),d(f(x),f(y))\}\\
			\le& |t-s|M+K_1d(x,y)+\max\{d(x,y),d(f(x),f(y))\},
		\end{align*}
		where $K_1>0$ is the constant in \eqref{eq:important-condition}. Thus, using \eqref{eq:important-condition} again, we obtain that  for each $n\in\mathbb N$, 
		\begin{align}
			\label{eq:qwe}
			\tilde{d}_n((x,t),(y,s))\,=&\, \max_{0\le i\le n-1}\left\{ \tilde{d}((f^i(x),t-S_i \rho(x)),(f^i(y),s-S_i \rho(y))\Biggr\} \notag\right.\\
			\le &\max\{M,K_1,1\}\cdot\max_{0\le i\le n-1}\left\{|t-s|+\sum_{j=0}^{i-1}|\rho(f^j(x))-\rho(f^j(y))|+\notag\right.\\
			& \,d(f^i(x),f^i(y))+\max\left\{d(f^i(x),f^i(y)),d(f^{i+1}(x),f^{i+1}(y))\right\}\notag\Biggr\}\\
			\le &\, 3\,\max\{M,K_1,1\}\cdot\max_{0\le i\le n-1}\Biggl\{|t-s|+K_1\max_{0\le j\le i-1}d(f^j(x),f^j(y))+\notag\\
			&\, d(f^i(x),f^i(y))+\max\{d(f^i(x),f^i(y)),d(f^{i+1}(x),f^{i+1}(y))\}\Biggr\}\notag\\
			\le &\, 3\max\{MK_1,K_1^2,K_1,1\}\cdot\max_{0\le i\le n}\{|t-s|+d(f^i(x),f^i(y))\}.
		\end{align}
		Denote $C=3\max\{M,MK_1,K_1^2,K_1,1\}$ and take $E\subset F$ a $(n,\varepsilon/2(1+C))$-spanning set for $n\in\mathbb N$. Now we consider 
		\[R=\left\{\ell\varepsilon/2(1+C): 0\leq \ell\leq\left\lfloor \frac{2(1+C)}{\varepsilon}\right\rfloor\right\}\]  
		and take 
		\[\bar E:=\{(x,t):x\in E \text{ and }t\in R\}\subset A.\]
		Now we notice that given $(y,s)\in A$, as $E$ is $(n,\varepsilon/2(1+C))$-spanning, there exists $x\in E$ so that $d_n(x,y)<\varepsilon/2(1+C)$. Choose $t\in R$ so that $|t-s|<\varepsilon/2(1+C)$. By \eqref{eq:qwe}, we have that 
		\begin{align*}
			\tilde{d}_{n-1}((x,t),(y,s))
			\le C\cdot(\varepsilon/(2(1+C))+\varepsilon/(2(1+C)))<\epsilon.
		\end{align*}
		So, $\bar E$ is a $(n-1,\varepsilon)$-spanning set for $\Psi_A$, contained in $\pi^{-1}_1(F)$ (with $\bar\mu_\varepsilon(\pi_1^{-1}(F))>1-\delta$) with cardinality less than $\left(\left\lfloor \frac{2(1+C)}{\varepsilon}\right\rfloor+1\right)\cdot \# E$. In particular,
		\[
		b_{\bar\mu_\varepsilon}(\varepsilon,\delta,n-1)\leq \left(\left\lfloor \frac{2(1+C)}{\varepsilon}\right\rfloor+1\right) \cdot b_{\nu_\varepsilon}(\varepsilon/2(1+C),\delta,n),
		\]
		(see Remark \ref{rmk_important} to recall the definition of $b_{\bar\mu_\varepsilon}(\cdot,\cdot,\cdot)$) which implies that 
		\begin{equation}\label{ineq-between-metric-entropies}
			h_{\bar\mu_\varepsilon}(\Psi_A,\varepsilon,\delta)\leq h_{\nu_\varepsilon}(f,\varepsilon/2(1+C),\delta).
		\end{equation}
		As $(\nu_\varepsilon)_\varepsilon\in\mathcal M(\mu)$, by the relation between $H(\mu)$ and $H^K_{\delta}(\mu)$ obtained in Section \ref{sec: conclusion of proof}, we conclude that $H((\mu_\rho)_A)\leq H(\mu)$.

		From now on our goal is to prove that $H((\mu_\rho)_A)=\frac{H(\mu_\rho)}{\mu_\rho(A)}$. Before we begin with this part of the proof we observe that the argument presented below works for $B\subset X_\rho$ for which $\mu_\rho(\partial B)=0$.
		Now we notice that, given  a partition $\hat{\xi}$ of $A$, we can construct a partition $\xi$ of $X$ by adjoining $A^c$.
		By the entropy formula of an induced transformation \cite[Chapter 6]{Petersen}, we have 
		\begin{equation}\label{eq:eqqe}
			h_{{\bar\nu}_A}(\Psi_A,\hat{\xi})=\frac{1}{\bar\nu(A)}h_{\bar\nu}(\Psi,\xi),
		\end{equation}
		for any $\bar\nu\in\mathcal M_\Psi(X_\rho)$ and $\xi$ of the previous type. Note that if a partition $\hat{\xi}$ of $A$ with $\mathrm{diam}(\hat{\xi})<\varepsilon$ for some $\varepsilon>0$, then there exists $n>0$ such that $\mathrm{diam}(\vee_{i=0}^{n-1}g_1^{-i}{\xi})<\varepsilon$. For any measurable partition $\xi$ of $X$ with $\mathrm{diam}(\hat{\xi})<\varepsilon$, let $\hat{\eta}=\{A\cap B:B\in\xi\}$ be a partition of $A$. By the discussion above, there exists $n=n(\eta)>0$ such that $\mathrm{diam}(\vee_{i=0}^{n-1}g_1^{-i}{\eta})<\varepsilon$. Then 
		\[h_{\mu}(\Psi,\xi)\ge h_{\mu}(\Psi,\eta)=h_{\mu}(\Psi,\vee_{i=0}^{n-1}g_1^{-i}{\eta}),\]
		where the last equation can be found in \cite[Proposition 9.3.2]{BS2002}. Therefore, we have that 
		\begin{equation}\label{eq:11.06}
			H(\mu_\rho)=\sup_{(\bar\mu_\varepsilon)_\varepsilon\in\mathcal M(\mu_\rho)}\limsup_{\varepsilon\to0}\frac{\inf_{|\hat{\xi}\cup\{A^c\}|<\varepsilon}h_{\bar\mu_\varepsilon}(\Psi,\xi)}{|\log\varepsilon|},
		\end{equation}
		where the infimum is taken over all   finite measurable partition $\hat \xi$ of  $A$.

		Now we that given $(\bar\mu_\varepsilon)_\varepsilon\in\mathcal M(\mu_\rho)$, as
		$\partial A=\{1\}\times X$, we have that $\mu_\rho(\partial A)=0$ and   $((\bar\mu_\varepsilon)_A)_\varepsilon\in\mathcal{M}((\mu_\rho)_A)$. Therefore,
		\begin{align*}       \frac{H(\mu_\rho)}{\mu_\rho(A)} &=\sup_{(\bar\mu_\varepsilon)_\varepsilon\in\mathcal M(\mu_\rho)}\limsup_{\varepsilon\to0}\frac{\inf_{|\xi|<\varepsilon}h_{\bar\mu_\varepsilon}(\Psi,\xi)}{\bar\mu_\varepsilon(A)|\log\varepsilon|}\\   &\le\sup_{(\bar\mu_\varepsilon)_\varepsilon\in\mathcal M(\mu_\rho)}\limsup_{\varepsilon\to0}\frac{\inf_{|\xi|<\varepsilon}h_{(\bar\mu_\varepsilon)_A}(\Psi_A,\xi)}{|\log\varepsilon|}\\
			&\leq H((\mu_\rho)_A).
		\end{align*}
		
		For the converse inequality we observe that given   $(\bar\nu_\varepsilon)_\varepsilon\in\mathcal{M}((\mu_\rho)_A)$, the sequence $(\bar\mu_\varepsilon)_\varepsilon$ defined as 
		\begin{align*}
			\bar\mu_\varepsilon(\cdot)=\mu_\rho(A)\bar\nu_\varepsilon(A\cap\cdot)+\mu_\rho(A^c)\bar\nu_\varepsilon(A^c\cap\cdot)
		\end{align*}
		satisfies the following:
		\begin{enumerate}
			\item[i.] $(\bar\mu_\varepsilon)_A=\bar\nu_\varepsilon$;
			\item[ii.] $(\bar\mu_\varepsilon)_\varepsilon\in\mathcal M(\mu_\rho)$.
		\end{enumerate}
		
		Applying  \eqref{eq:11.06} and \eqref{eq:eqqe}, we obtain $\displaystyle\frac{H(\mu_\rho)}{\mu_\rho(A)}\geq H((\mu_\rho)_A)$ and hence
		\begin{align}
			\label{eq:eqeqotaiio}
			\displaystyle\frac{H(\mu_\rho)}{\mu_\rho(A)}= H((\mu_\rho)_A)
		\end{align}

		We now consider the case of a general roof function.   Given $\rho \colon X\to (0,\infty)$, let us consider $\rho+1$. Then 
		\[H(\mu_{\rho+1})=\frac{H(\mu)}{\int \rho\ d\mu+1}.\]
		Since $X_\rho\subset X_{\rho+1}$, $\mu_{\rho+1}(\partial X_\rho)=\mu_{\rho+1}(\{(x,\rho(x)):x\in X\})$ and $\rho $ is a continuous function, by \eqref{eq:eqeqotaiio} 
		\[H((\mu_{\rho+1})_{X_\rho})=\frac{H(\mu_{\rho+1})}{\mu_{\rho+1}(X_\rho)}.\]
		Since $(\mu_{\rho+1})_{X_\rho}=\frac{\mu_{\rho+1}(\cdot \cap X_{\rho})}{\mu(X_{\rho)}}$, we deduce that
		\[(\mu_{\rho+1})_{X_\rho}=\mu_\rho\]
		and 
		\[\mu_{\rho+1}(X_\rho)=\frac{\mu(X_\rho)}{\mu(X_{\rho+1})}=\frac{\int \rho\ d\mu}{\int \rho+1\ d\mu}.\]
		Thus, 
		\[H(\mu_\rho)=\frac{\int \rho +1d\mu}{\int \rho d\mu}\cdot \frac{1}{\int \rho+1 d\mu}H(\mu)=\frac{1}{\int \rho d\mu}\cdot H(\mu).\]
		The proof is completed.
	\end{proof}
	
	\begin{remark}\label{rem:Aformula}
		We observe that, as a consequence of the proof of the previous theorem, we get that for any continuous map $h:X\to X$ acting on a compact metric space $(X,d)$,
		\[H(\nu_B)=\frac{H(\nu)}{\nu(B)},\]
		for any $\nu\in\mathcal M_h(X)$ and $\nu$-measurable set $B$ with $\nu(B)>0$ and  $\nu(\partial B)=0$, where $\nu_B$ is the restriction of $\nu$ to $B$.
		
		Moreover, our results guarantee that given $\alpha\in[0,\infty)$, there exists a flow, acting on a compact metric space, with metric mean dimension greater or equal to $\alpha$. 
	\end{remark}
	
	\begin{theorem}
		Let $(X,d)$ be a compact metric space and $f\colon X\to X$ be a homeomorphism with the specification property.  Let $\rho\colon X\to(0,\infty)$ be a continuous function satisfying \eqref{eq:important-condition} and $(X_\rho,\Psi)$ be the corresponding suspension flow over $X$. Suppose $X_\rho$ is endowed with the Bowen-Walters metric $\tilde d$ induced by $d$.   
		If $\Phi\colon  X_\rho\to \mathbb R $ is a continuous function then
		\begin{align*}
			\mathrm{\overline{mdim}_M^B}\,(X_\rho(\Phi,\alpha),\Psi,\tilde d)\geq  \sup\left\{H(\mu_\rho):\mu_{\rho}\in \mathcal M_\Psi(X_\rho)\text{ and }\int_{X_\rho}\Phi\ d\mu_\rho=\alpha\right\}.
		\end{align*}
	\end{theorem}
	\begin{proof}
		Let $\beta >0$ be the unique solution of the equation $\mathrm{\overline{mdim}_M^B}\,(X_\rho(\Phi,\alpha),f, d,-t\rho)=0$. Then, by Theorem \ref{thm2020202}, we have that $\mathrm{\overline{mdim}_M^B}\,(X_\rho(\Phi,\alpha),\Psi,\tilde d)\geq \beta$. Now, letting $\varphi$ be the map associated to $\Phi$ as in the beginning of the section, it follows by Theorem \ref{thm:10101010}  that
		\[
		\sup\left\{H(\mu)-\beta\int_X\rho \ d\mu : \mu\in\mathcal M_f(X)\text{ and }\frac{\int_{X}\varphi \ d\mu}{\int_X \rho \ d\mu}=\alpha\right\}=0.
		\]
		Consequently, if $\mu\in\mathcal M_f(X)$ satisfies $\int\varphi\ d\mu\slash \int \rho\ d\mu=\alpha$, then $\beta\geq H(\mu)\slash \int\rho\ d\mu$. Therefore, from \eqref{eq: metric entropy suspension} it follows that
		\begin{align*}
			\beta&\geq\sup\left\{\frac{H(\mu)}{\int\rho\ d\mu}: \mu\in\mathcal M_f(X)\text{ and }\frac{\int_{X}\varphi \ d\mu}{\int_X \rho \ d\mu}=\alpha\right\}\\
			&=\sup\{ H(\mu_\rho ):\mu\in\mathcal{M}_\Psi(X_\rho)\text{ and }\int \Phi\ d\mu=\alpha\}.
		\end{align*}
		which completes the proof of the theorem. 
	\end{proof}
	
	\begin{remark}
		In general, one can not expect to get an equality in the previous theorem. In fact,  as by \cite{Gu} there exists a  minimal topological dynamical system $f:X\to X$ with compatible metric
		and a roof function $\rho : X \to (0,\infty)$ such that $\mathrm{\overline{mdim}_M^B}\,(X,f,  d)=0$ but $\mathrm{\overline{mdim}_M^B}\,(X_\rho,\Psi_1, \tilde d)>0$, and by \eqref{eq: metric entropy suspension}
		\begin{align*}
			0=  \mathrm{\overline{mdim}_M^B}\,(X,f, d)&=\sup \left\{H(\mu):\mu\in \mathcal M_f(X)\right\}
			\\
			&=\sup\left\{H(\mu_\rho):\mu_{\rho}\in \mathcal M_\Psi(X_\rho)\right\}.
		\end{align*}
		Thus,
		\[\sup\left\{H(\mu_\rho):\mu_{\rho}\in \mathcal M_\Psi(X_\rho)\right\}<\mathrm{\overline{mdim}_M^B}\,(X_\rho
		,\Psi,\tilde d).\]
	\end{remark}

	
	\medskip{\bf Acknowledgements.}
	The authors express gratitude to Prof. Gutman for providing many meaningful suggestions.

	\section{Declarations}
	\textbf{Ethical Approval:} not applicable.
	
	\textbf{Funding:} L.~Backes was partially supported by a CNPq-Brazil PQ fellowship under Grant No. 307633/2021-7. 
	
	\textbf{Data Availability:} No data sets were generated or analysed during the current study.
	
	\textbf{Declarations:} Conflict of interest: No potential conflict of interest was reported by the authors.
	\bibliographystyle{acm}

\end{document}